\documentclass[12pt,reqno]{amsart}
\usepackage{amssymb,amsmath,tabularx,mathrsfs,mathbbol,yfonts,upgreek,hyperref}
\usepackage{amsthm,verbatim,comment}
\usepackage{geometry}
\usepackage{stmaryrd}
\usepackage{paralist}
\usepackage[all]{xy}
\usepackage{mathdots}
\usepackage{tikz}
\usepackage{ytableau}
\usetikzlibrary{arrows,matrix,positioning,fit}
\usepackage{graphicx}
\allowdisplaybreaks
\usepackage{nicematrix}
\NiceMatrixOptions{
code-for-first-row = \color{red} ,
code-for-last-row = \color{red} ,
code-for-first-col = \color{blue} ,
code-for-last-col = \color{red}
}

\DeclareSymbolFontAlphabet{\mathbb}{AMSb}
\DeclareSymbolFontAlphabet{\mathbbl}{bbold}

\newtheorem{thm}{Theorem}[section]
\newtheorem{lem}[thm]{Lemma}
\newtheorem{prop}[thm]{Proposition}
\newtheorem{cor}[thm]{Corollary}

\theoremstyle{definition}
\newtheorem{defn}[thm]{Definition}

\newtheorem{eg}[thm]{Example}
\newtheorem{rem}[thm]{Remark}
\theoremstyle{remark}
\newtheorem*{rem*}{Remarks}

\renewcommand{\mod}{\mathrm{mod}\ }

\newcommand{\sgn}{\mathrm{sgn}}

\newcommand{\sym}[1]{\mathfrak{S}_{#1}}
\newcommand{\e}[1]{\overline{e_{#1}}}
\renewcommand{\t}[1]{\overline{\mathfrak{t}_{#1}}}

\newcommand{\A}{\mathbb{A}}

\newcommand{\N}{\mathbb{N}}

\newcommand{\B}{\mathcal{B}}

\newcommand{\F}{\mathbb{F}}
\newcommand{\U}[2]{{\mathrm{U}^\wedge_{{#1}}({#2})}}
\newcommand{\ind}[2]{{#1}{\uparrow^{#2}}}
\newcommand{\res}[2]{{#1}{\downarrow_{#2}}}
\newcommand{\rad}{\mathrm{Rad}}
\DeclareMathOperator{\rank}{rank}

\numberwithin{equation}{section}
\newcommand{\rk}[2]{V^\#_{#1}({#2})}
\newcommand{\Nom}{\mathrm{N}}

\begin{document}
\title[The rank varieties]{Small modules with interesting rank varieties}

\author{Kay Jin Lim}
\address[K. J. Lim]{Division of Mathematical Sciences, Nanyang Technological University, SPMS-MAS-05-16, 21 Nanyang Link, Singapore 637371.}
\email{limkj@ntu.edu.sg}

\author{Jialin Wang}
\address[J. Wang]{Division of Mathematical Sciences, Nanyang Technological University, SPMS-MAS-04-15, 21 Nanyang Link, Singapore 637371.}
\email{wang1483@e.ntu.edu.sg}

\begin{abstract} This paper focuses on the rank varieties for modules over a group algebra $\F E$ where $E$ is an elementary abelian $p$-group and $p$ is the characteristic of an algebraically closed field $\F$. In the first part, we give a sufficient condition for a Green vertex of an indecomposable module to contain an elementary abelian $p$-group $E$ in terms of the rank variety of the module restricted to $E$. In the second part, given a homogeneous algebraic variety $V$, we explore the problem on finding a small module with rank variety $V$. In particular, we examine the simple module $D^{(kp-p+1,1^{p-1})}$ for the symmetric group $\sym{kp}$.
\end{abstract}

\thanks{We thank the referee for numerous suggestion especially for a shorter proof of Lemma \ref{L: rkinduction}. The first author is supported by Singapore Ministry of Education AcRF Tier 1 grant RG17/20.}

\maketitle

\section{Introduction}

Let $\F$ be an algebraically closed field with positive characteristic $p$. Most group algebras of a finite $p$-group in characteristic $p$ have wild representation theory. As such, one does not hope to classify the indecomposable modules up to isomorphism. Instead, various techniques have been brought in to study the representations without the intention of making such classification. In this paper, we mainly focus on the complexity \cite{AE81,AE82}, Green vertex \cite{Green59} and rank variety \cite{carlson} for modules. These notions are interrelated. For example, if $E$ is an elementary abelian $p$-subgroup of rank $k$ of a finite group $G$ and $M$ is an indecomposable $\F G$-module with the rank variety of $\res{M}{E}$ the same as $\F^k$, then $E$ is a subgroup of a Green vertex of $M$. The first result in our paper improves such relation (see Theorem \ref{T: Green}).

In the second part of this paper, we study the exterior power of the natural simple $\F\sym{n}$-module $D^{(n-1,1)}$. Our study is motivated by various sources. For simplicity, for each homogeneous algebraic variety $V$, in \cite{Carlson84}, Carlson constructed a module with rank variety $V$. The dimension of the module obtained in this manner is `pretty large'. In \cite[\S12.7]{benson2016}, with the help of the $L_\zeta$-technology, Benson constructed an $\F E$-module of dimension 36 with rank variety \[p_3=x_1^2x_2^2+x_1^2x_3^2+x_2^2x_3^2\] when $p=3$ and $E=(C_3)^3$. This is smaller than the Specht module $S^{(3^3)}$ (which has dimension 42) with the same rank variety upon restriction to the elementary abelian $3$-subgroup generated by 3 disjoint 3-cycles (see \cite{L09}). In \cite[Appendix B, Problem 16]{benson2016}, for a homogeneous algebraic variety $V$, Benson asked for the smallest dimension $d_V$ of an $\F E$-module with rank variety $V$. The problem is obviously extremely difficult to answer in general. Nevertheless, we wish to offer some insight to this problem from the viewpoint of the representations of symmetric groups. More precisely, we want to find naturally occurring representations from the symmetric groups with interesting varieties. Apart from, more generally, the Specht module $S^{(p^p)}$, the other examples are the simple module $D^{(n-1,1)}$ in $p=2$ (see \cite{Jiang21}) and some basic spin modules (see \cite{Uno}).  In \cite[Theorem 4.9]{LT13}, Tan and the first author proved that all simple modules $D^\lambda$ for the symmetric groups belonging to the weight 2 block have complexity 2 except $\lambda=(p+1,1^{p-1})$. This suggested the study of the simple module $D(p-1):=D^{(kp-p+1,1^{p-1})}\cong \bigwedge^{p-1}D^{(kp-1,1)}$ for $k\geq 2$. Notice that, when $p\nmid n$, for all $1\leq r\leq n-1$, $\bigwedge^rD^{(n-1,1)}\cong \bigwedge^rS^{(n-1,1)}\cong S^{(n-r,1^r)}$ which is a simple Specht module. As such, the variety of $\bigwedge^rD^{(n-1,1)}$ is `uninteresting' as shown in \cite{DL17}. When $p\mid n$ and $1\leq r<p-1$, we have $p\nmid \dim_\F D^{(n-r,1^r)}$ and its variety is again `uninteresting'. As such, the module $D(p-1)$ is the `smallest' we could have picked.

Suppose that $k\not \equiv 1(\mod p)$ {and $p$ is odd}. In this paper, we prove that $D(p-1)$ has complexity $k-1$. More precisely, we show that the rank variety of $D(p-1)$ restricted to the largest elementary abelian $p$-group $E_k$ is the algebraic set $V(p_k)$ where $p_k$ is given as in Equation \ref{Eq: pk}. We consider the module is rather small in the following sense. When $p=3$ and $k\geq 3$, Carlson's result (see Theorem \ref{T: dimension}) shows that any module with the rank variety $V(p_k)$ has dimension at least $6(k-1)$ and divisible by $3$. On the other hand, $\dim D(2)={3k-2\choose 2}=\frac{3(k-1)(3k-2)}{2}$. Its ratio with Carlson's bound is \[\frac{\dim D(2)}{6(k-1)}=\frac{3k-2}{4}\] which is a polynomial of degree 1 in $k$. The particular case when $k=3$, the module $M:=\res{D^{(7,1,1)}}{E_3}$ has the rank variety $V(p_3)$ as above and has dimension 21 which is strictly smaller than the dimensions of the modules we have discussed in the previous paragraph, which are 36 and 42. Using Magma \cite{Magma}, it turns out also that $M$ is indecomposable and we have not found a strictly smaller module than $M$ and yet supporting $p_3$.

The methods we have employed in our computation include the notion of generic and maximal Jordan types of modules, their relation with the Schur functor and brute-force calculation. We believe that the condition $k\not \equiv 1(\mod p)$ is unnecessary and further conjecture that $\bigwedge^rD^{(kp-1,1)}$ has interesting variety when $r\equiv p-1(\mod p)$.

In the next section, we collate the basic knowledge we shall need in this paper. In Section \ref{S: Green Vertex}, {we prove our result Theorem \ref{T: Green} regarding the rank variety and Green vertex}. In Section \ref{S: D1}, we compute the generic Jordan type and the maximal Jordan set of the module $\res{D(1)}{E_k}$ (see Theorem \ref{T: jtD}) as a preparation for the discussion in the next section. In the final section, we show that the rank variety of $\res{D(p-1)}{E_k}$ is the algebraic set $V(p_k)$ as in Theorem \ref{T: main thm}.

\section{Preliminary}

For the basic knowledge required in this article, we refer the reader to the references \cite{Benson1,Benson2,benson2016,James78,Shafarevich}. Throughout, $\F$ is an algebraically closed field of positive characteristic $p$.

\subsection{Modules for finite groups} Let $G$ be a finite group. All the $\F G$-modules we consider in this paper are finite-dimensional over $\F$. We use the notations $\ind{}{}$ and $\res{}{}$ for the induction and restriction of modules for finite groups.  The direct sum and tensor product of two $\F G$-modules $M,N$ are denoted by $M\oplus N$ and $M\otimes_\F N$ (or simply $M\otimes N$) respectively. By abuse of notation, the trivial $\F G$-module is also denoted as $\F$ (or $\F_G$ if we wish to emphasis the group $G$). If $N$ is a direct summand of $M$, that is $M\cong N\oplus N'$ for another $\F G$-module $N'$, we write $N\mid M$.

Let $M$ be an indecomposable $\F G$-module. A Green vertex $Q$ of $M$ is a minimal subgroup of $G$ such that $M\mid \ind{S}{G}$ for some indecomposable $\F Q$-module $S$. In this case, $S$ is called an $\F Q$-source of $M$ and $Q$ is necessarily a $p$-subgroup.  Furthermore, all green vertices of $M$ are conjugate in $G$.

Let $H,K$ be subgroups of $G$ and $N$ be an $\F K$-module. We have the Mackey's formula \[\res{\ind{N}{G}}{H}\cong \bigoplus_{HgK} \ind{\res{{}^gN}{H\cap {}^gK}}{H}\] where the sum runs over a complete set of double coset representatives of $(H,K)$ in $G$.

For a finite group $G$ and an $\F G$-module $M$, the complexity of $M$ is the rate of growth of a projective resolution of $M$ and is denoted as $c_G(M)$. Moreover, if $\mathscr{E}$ is a set of representatives of the maximal elementary abelian $p$-subgroups of $G$ up to conjugation, we have
\begin{equation}\label{Eq: complexity}
c_G(M)=\max_{E\in \mathscr{E}}c_E(\res{M}{E}).
\end{equation}

\subsection{The representations of symmetric groups}\label{SS: sym} For a finite set $A$, we denote $\sym{A}$ the symmetric group permuting the elements in $A$. For a natural number $n$, we let \[\sym{n}=\sym{\{1,2,\ldots,n\}}.\] Let $k$ be a positive integer and $n=kp$. In this paper, we are interested in modules for the symmetric group $\sym{kp}$ restricted to the elementary abelian $p$-subgroup \[E_k=\langle g_1,\ldots,g_k\rangle\] where, for each $1\leq i\leq k$, $g_i$ is the $p$-cycle $((i-1)p+1,(i-1)p+2,\ldots,ip)$. Notice that, when $p$ is odd, $E_k$ is, up to conjugation, the only elementary abelian $p$-subgroup of $\sym{kp}$ with rank $k$ and the remaining have ranks strictly less than $k$. Notice that \[\mathrm{N}_{\sym{kp}}(E_k)/\mathrm{C}_{\sym{kp}}(E_k)\cong \F_p^\times\wr \sym{k}.\]

A partition $\lambda$ of $n$ is a sequence of positive integers $(\lambda_1,\dots,\lambda_k)$ such that $\lambda_1\geq \lambda_2\geq \cdots\geq\lambda_k$ and $\lambda_1+\cdots+\lambda_k=n$. It is a hook partition if $\lambda_2=\cdots=\lambda_k=1$. It is $p$-regular if $\lambda$ does not contain $p$ parts of the same size. The Young diagram of $\lambda$ is the depiction of the set $\{(i,j): 1\leq i\leq k, 1\leq j\leq \lambda_i\}$ and an element in the set (or diagram) is called a node. A $\lambda$-tableau is an array obtained by assigning the nodes in the Young diagram of $\lambda$ by the numbers $1,2,\dots,n$ with no repeats. We say that a $\lambda$-tableau is standard if the numbers are increasing both in each row from left to right and in each column from top to bottom. Given another partition $\mu$ of $n$, $\lambda$ dominates $\mu$ if, for all $r$, \[\sum^r_{i=1}\lambda_i\geq \sum^r_{i=1}\mu_i\] where we let $\lambda_i=0$ if $i>k$ and similar for $\mu$. In this case, we write $\lambda\unrhd \mu$.

Fix a partition $\lambda$ of $n$. The symmetric group $\sym{n}$ has a natural action on the $\lambda$-tableaux by permuting the numbers. The row (respectively, column) stabilizers $R_t$ (respectively, $C_t$) of a $\lambda$-tableau $t$ is the subgroup of $\sym n$ consisting of the elements fixing the rows (respectively, columns) of $t$ setwise. A $\lambda$-tabloid $\{t\}$ is the equivalence class containing $t$ under the equivalence relation defined by: for $\lambda$-tableaux $t,t'$, $t\sim t'$ if and only if $t=\pi t'$ for some $\pi\in R_t$. As such, the natural action of $\sym n$ acts on the $\lambda$-tabloids by permuting the numbers as well.

Let $t$ be a $\lambda$-tableau. We define the polytabloid \[e_t=\sum_{\sigma \in C_t} (\sgn\sigma)\sigma\{t\}.\] We say that $e_t$ is standard if $t$ is standard. The Specht module $S^\lambda$ is the $\F\sym n$-module which is, as a vector space, spanned by the $\lambda$-polytabloids. The set of standard $\lambda$-polytabloids forms a characteristic free basis for $S^\lambda$ and its dimension is given by the hook formula. In the case $\lambda$ is $p$-regular, $S^\lambda$ has a simple head $D^\lambda$. Moreover, the set of $D^\lambda$ where $\lambda$ runs over all $p$-regular partitions of $n$ gives a complete set of simple $\F\sym n$-modules up to isomorphism.

Let $n>2$ and consider the natural simple $\F\sym{n}$-module $D(1):=D^{(n-1,1)}$. For any $r\leq \dim_\F(D(1))$, define the $r$th exterior power \[D(r):=\bigwedge^rD(1).\] In fact, the surjection $S^{(n-1,1)}\twoheadrightarrow D(1)$ induces a surjection $S^{(n-r,1^r)}\cong \bigwedge^r S^{(n-1,1)}\twoheadrightarrow D(r)$ where the isomorphism can be found in, for example, \cite[Proposition 2.3(a)]{MZ07}. In the case when $p$ is odd, we have that $D(r)$ is again a simple $\F\sym{n}$-module (see \cite{Danz07}). In this case, using \cite[6.3.59]{JK} and \cite{Peel}, we have $D(r)\cong D^{(n-r,1^r)^R}$ where $R$ denotes the $p$-regularisation of a partition; in particular, \[D(p-1)\cong D^{(n-p+1,1^{p-1})}.\]

Suppose that $n=kp$ and $p$ is an arbitrary prime. In this case, $S^{(kp-1,1)}$ has composition factors $D^{(kp-1,1)}$ and $D^{(kp)}\cong \F$ from top to bottom. For each $i=1,\ldots,kp$, let $\mathfrak{t}_i$ denote the $(kp-1,1)$-tabloid with $i$ in the second row and, if $i\neq 1$, let $e_i=e_{\mathfrak{t}_i}=\mathfrak{t}_i-\mathfrak{t}_1$. For $\sigma \in \sym {kp}$, we have $\sigma e_i=\mathfrak{t}_{\sigma i}-\mathfrak{t}_{\sigma 1}$. As such, $\{e_i:i=2,\ldots,kp\}$ is a basis for $S^{(kp-1,1)}$ with the trivial submodule $D^{(kp)}$ spanned by $\sum^{i=kp}_{i=2}e_i$. Let $\e i=e_i+D^{(kp)}$. Thus we obtain \[\{ \e i:  i=3,\ldots,kp\}\] a basis for $D(1)=D^{(kp-1,1)}$. In this case, for $1\leq r\leq kp-2$, \[\dim_\F D(r)=\binom{kp-2}{r}.\] In particular, $p\nmid \dim_\F D(r)$ for all $r=1,\ldots,p-2$.

%We use $D$ to denote the simple $\F\sym {kp}$-mod $D^{(kp-1,1)}$ which is also called the natural simple module, $D(p-1)$ to denote the simple $\F\sym {kp}$-mod $D^{(kp-p+1, 1^{p-1})}$. In particular, $D(p-1)\cong\wedge^{p-1}D$. In this article, we will look at $D$ and $D(p-1)$ closely. It is known that if $p\mid n$, then $D\cong S^{(n-1,1)}/S^{(n)}$. For $i=2,3,\cdots,n$, let $e_i$ denote the polytabloid corresponding to the standard tableau with $i$ in the second row. Then the basis of $S^{(n-1,1)}$ is given by $\{e_i: i=2,3,\cdots,n \}$. In particular, $S^{(n)}$ is viewed as a $\F\sym n$-submodule of $S^{(n-1,1)}$ spanned by $\{\sum_{i=2}^{i=n}e_i\}$. Let $\e i$ denote $e_{i}+S^{(n)}$ in the quotient. Thus we obtain a basis for $D$, which is
%\[\{ \e i:  \text{for } i=3,\cdots,kp\}.
%\]

\subsection{Algebraic variety}

Consider the polynomial ring $R:=\F[x_1,\ldots,x_k]$. For any ideal $I$ of $R$, let $V(I)$ denote the algebraic set. The affine space $\A^k(\F)$ is a noetherian topological space with the algebraic sets as the closed sets. For the algebraic variety $V(I)$, we let $\dim V(I)$ denote its dimension, that is the supremum of the lengths $d$ of the chains of distinct irreducible closed subsets of $V(I)$ of the form $V_d\subset \cdots\subset V_1\subset V_0$. Furthermore, by definition, $\dim V(I)$ is equal to the maximum of the dimensions of the irreducible components of $V(I)$.  In the case when $I$ is prime, i.e., $V(I)$ is irreducible, $\dim V(I)$ is the Krull dimension of $R/I$. We shall need the following.

\begin{thm}[{\cite[Chapter 1, \S6, Theorem 1]{Shafarevich}}]\label{T: irred variety dim} Suppose that $V$ is irreducible and $W\subseteq V$. If $\dim W=\dim V$ then $W=V$.
\end{thm}

Suppose that $k$ is an integer at least $2$. Throughout the paper, we denote
\begin{align}\label{Eq: pk}
p_k:=\sum_{i=1}^{k}\left (x_1\cdots \widehat{x_i}\cdots x_k\right )^{p-1}
\end{align} where $\widehat{x_i}$ indicates that $x_i$ does not appear as a factor in the corresponding summand. By convention, we denote $p_1=1$. When $k=2$, we have \[p_2=x_1^{p-1}+x_2^{p-1}=\prod_{\lambda\in I}(x_1-\lambda x_2)\] where $I$ is the complete set of $(p-1)$th roots of $-1$ in $\F$.  However, when $k\geq 3$, the polynomial $p_k$ is irreducible as given below.

\begin{lem}\label{irred}
If $k\geq 3$, the polynomial $p_k$ is irreducible and the variety $V(p_k)$ is irreducible of dimension $k-1$.
\end{lem}
\begin{proof}
We argue by induction on $k$. When $k=3$, $p_3=x^{p-1}y^{p-1}+x^{p-1}z^{p-1}+y^{p-1}z^{p-1}$.
Suppose that $p_3=gh$ for some $g,h \in \F[x,y,z]$. Since $p_3$ is homogeneous, both $g,h$ are homogeneous, say of degrees $r,s$ respectively. Divide both sides throughout by $z^{2p-2}$. Let $X=\frac{x}{z}$, $Y=\frac{y}{z}$, $P_3(X,Y)=\frac{1}{z^{2p-2}}p_3$, $G(X,Y)=\frac{1}{z^{r}}g$ and $H(X,Y)=\frac{1}{z^s}h$. We have
\begin{align*}
G(X,Y)H(X,Y)=P_3(X,Y)&=X^{p-1}Y^{p-1}+X^{p-1}+Y^{p-1}\\
&=X^{p-1}(1+Y^{p-1})+Y^{p-1}.
\end{align*}
Let $G(X,Y)=a_0X^0+a_1X^1+\ldots+a_iX^i$ and $H(X,Y)=b_0X^0+b_1X^1+\ldots+b_{p-1-i}X^{p-1-i}$ such that all the coefficients $a_0,\ldots,a_i,b_0,\ldots,b_{p-1-i}$ are in $\F[Y]$ and, without loss of generality, $0\leq i\leq p-1-i$. In particular, we have $a_ib_{p-1-i}=1+Y^{p-1}$. Since $\F$ is algebraically closed, $1+Y^{p-1}$ has $p-1$ distinct roots $\lambda_1,\lambda_2,\ldots, \lambda_{p-1}$ in $\F$ and \[1+Y^{p-1}=\prod_{j=1}^{p-1} (Y-\lambda_j).\] Thus, for each $1\leq j\leq p-1$, $Y-\lambda_j$ is either a factor of $a_i$ or a factor of $b_{p-1-i}$ (and not both). As such, we must have either $a_i(\lambda_j)\neq 0$ or $b_{p-1-i}(\lambda_j)\neq 0$.  On the other hand, for each $1\leq j\leq p-1$, $P_3(X,\lambda_j)=\lambda_j^{p-1}=-1\in \F$. This implies
\[\deg(P_3(X,\lambda_j))=\deg(G(X,\lambda_j))+\deg(H(X,\lambda_j))=0,\]
i.e., $\deg(G(X,\lambda_j))=\deg(H(X,\lambda_j))=0$ for any $\lambda_j$. If $Y-\lambda_j$ is a factor of $a_i$ then $b_{p-1-i}(\lambda_j)\neq 0$ and $\deg(H(X,\lambda_j))=p-1-i=0$. So $i=p-1>p-1-i=0$, a contradiction. This shows that $b_{p-1-i}=\frac{1}{\mu}(1+Y^{p-1})$ and $a_i=\mu\neq 0$ in $\F$. But $0=\deg(G(X,\lambda_j))=i$. Hence $G(X,Y)=\mu\neq 0$. In particular, $g(x,y,z)=\mu z^r$. But $z\nmid p_3$. Therefore $r=0$, i.e., $g=\mu\in \F$. Thus, $p_3$ is irreducible.

Suppose now that $p_{k-1}(x_1,x_2,\ldots, x_{k-1})\in \F[x_1,x_2,\ldots, x_{k-1}]$ is irreducible for some $k\geq 4$. Let
\[D=\F[x_1,x_2,\ldots, x_{k-1}]/(p_{k-1})\] which is an integral domain. For a polynomial $f\in \F[x_1,\ldots,x_k]$, let $f=f_0x_k^0+\cdots+f_mx_k^m$ where $f_0,\ldots,f_m\in \F[x_1,x_2,\ldots, x_{k-1}]$, we write \[\bar{f}=\overline{f_0}x_k^0+\cdots+\overline{f_m}x_k^m\in D[x_k]\] where $\overline{f_i}=f_i+(p_{k-1})\in D$. Suppose that $p_k=gh$ for some $g,h\in \F[x_1,x_2,\ldots, x_k]$. Since $p_k$ is homogeneous, both $g,h$ are homogeneous. Notice that $\overline{g}\overline{h}=\overline{p_k}=\overline{(x_1x_2\cdots x_{k-1})^{p-1}}$ and so \[\deg(\bar{g})+\deg(\bar{h})=\deg(\overline{p_k})=0.\] As such, $\deg(\bar g)=0=\deg(\bar h)$. Thus, there exist $s,s'\in \F[x_1,x_2,\ldots, x_k]$ and $t,t'\in \F[x_1,x_2,\ldots, x_{k-1}]$ such that
\begin{align*}
g&=p_{k-1}s+t,\\
h&=p_{k-1}s'+t',
\end{align*}
where $t=\frac{1}{\mu}x_1^{i_1}x_2^{i_2}\cdots x_{k-1}^{i_{k-1}}$ and $t'=\mu x_1^{p-1-i_1}x_2^{p-1-i_2}\cdots x_{k-1}^{p-1-i_{k-1}}$ for some integers $0\leq i_1,i_2,\ldots,i_{k-1}\leq p-1$ and $0\neq\mu\in \F$. We claim that $\deg(g)=\deg(p_{k-1})+\deg(s)$ if $s\neq 0$. Let $d=\deg(s)$ and $s_d$ be the $d$th homogeneous component of $s$. Then \[p_{k-1}s_d+t=\frac{1}{\mu}x_1^{i_1}x_2^{i_2}\cdots x_{k-1}^{i_{k-1}}+\sum_{i=1}^{k-1}x_1^{p-1}\cdots \widehat{x_i^{p-1}}\cdots x_{k-1}^{p-1}s_d\neq 0\] and therefore $\deg(g)=\deg(p_{k-1})+\deg(s_d)=\deg(p_{k-1})+d$ as required. If $ss'\neq 0$, let $d'=\deg(s')$, we get \[(k-1)(p-1)=\deg(p_k)=\deg(gh)=2\deg(p_{k-1})+d+d'\geq 2(k-2)(p-1),\] i.e., $3\geq k$, a contradiction. Without loss of generality, let us assume $s'=0$ (and hence $s\neq 0$). Then $g=p_{k-1}s+t$ and $h=t'\in \F[x_1,x_2,\ldots, x_{k-1}]$. On the other hand,
\[x_k^{p-1}p_{k-1}+(x_1x_2\cdots x_{k-1})^{p-1}=p_{k-1}st'+tt',
\]
i.e., $st'=x_k^{p-1}$. So $h=t'=\mu$ and $p_k$ is irreducible. Since $p_k$ is irreducible, $(p_k)$ is prime and the variety $V(p_k)$ is irreducible of dimension $k-1$.
\end{proof}

\subsection{Rank variety for a module}

We now review the rank variety for a module as introduced by Carlson \cite{carlson}.  Let $E=\langle g_1,\ldots,g_k\rangle$ be an elementary abelian $p$-group of rank $k$ with the generators in the order $g_1,\ldots,g_k$. For each $i=1,\ldots,k$,  denote $X_i=g_i-1\in \F E$. The set $\{X_1+J^2,\dots,X_k+J^2\}$ forms a basis for $J/J^2$ where $J=\rad(\F E)$ is the Jacobson radical of $\F E$. Let $\A^k(\F)$ denote the affine $k$-space over $\F$ consisting of $k$-tuples $\alpha=(\alpha_1,\dots,\alpha_k)$ with each $\alpha_i\in \F$. For $0\neq \alpha=(\alpha_1,\dots,\alpha_k)\in\A^k(\F)$ and let \[X_\alpha=\alpha_1X_1+\cdots+\alpha_kX_k\in \F E\] and $u_\alpha=1+X_\alpha$. Since we are in characteristic $p$, we have $X_\alpha^p=0$ and $u_\alpha^p=1$, that is $\langle u_\alpha\rangle$ is a cyclic subgroup of $(\F E)^\times$ of order $p$ and it is called a cyclic shifted subgroup of $E$.

Let $M$ be an $\F E$-module.  We denote the consideration of $M$ as $\F\langle u_\alpha\rangle$-module by $\res{M}{\langle u_\alpha\rangle}$.  The rank variety $\rk{E}{M}$ of $M$ is defined as
\[\rk{E}{M}=\{0\}\cup\{0\neq \alpha\in \A^k(\F): \text{$\res{M}{\langle u_\alpha\rangle}$ is not free}\}.\]
{Up to isomorphism, the rank variety $V^\#_E(M)$ is independent of the choice of generators for $E$ (see \cite[Theorem 6.5]{carlson}).} By slight abuse of notation, for an $\F G$-module $M$ and an elementary abelian $p$-subgroup $E$ of $G$, we write $V^\#_E(M)$ for $V^\#_E(\res{M}{E})$. The rank variety of a projective $\F E$-module is described as follows.

\begin{lem}[Dade's lemma {\cite[Lemma 11.8]{Dade78}}]\label{L: Dade} Let $M$ be an $\F E$-module. Then $M$ is projective if and only if $V^\#_E(M)=\{0\}$.
\end{lem}

Given a short exact sequence $0\to M'\to M\to M''\to 0$ of $\F E$-modules, by definition, we have $V^\#_E(M)\subseteq V^\#_E(M')\cup V^\#_E(M'')$. As a consequence, we get the following lemma.

\begin{lem}\label{L: filtra} Let $M$ be an $\F E$-module and consider a filtration of $M$ with subquotients $Q_1,\ldots,Q_r$. Then \[V^\#_E(M)\subseteq \bigcup_{i=1}^r V^\#_E(Q_i).\]
\end{lem}

Moreover, we have the following basic properties regarding rank varieties.

\begin{thm}\label{T: basic rank}
Let $M$ and $N$ be $\F E$-modules. Then
\begin{enumerate}[(i)]
\item $\rk{E}{M}$ is a closed homogeneous subvariety of $\rk{E}{\F}=\A^k(\F)$,
\item the dimension of $\rk{E}{M}$ is equal to the complexity $c_E(M)$,
\item $\rk{E}{M\oplus N}=\rk{E}{M}\cup\rk{E}{N}$ and $\rk{E}{M\otimes N}=\rk{E}{M}\cap\rk{E}{N}$,
\item if $M$ is indecomposable, then the projective variety $\overline{V^\#_E(M)}$ is connected,
%\item if $G$ is another elementary abelian $p$-group containing as a subgroup $E$, then $\dim V^\#_G(\ind{M}{G})=\dim V^\#_E(M)$.
\end{enumerate}
\end{thm}

Obviously, when $p\nmid \dim_\F M$, $\res{M}{\langle u_\alpha\rangle}$ is never free and therefore $V^\#_E(M)=\A^k(\F)$. The following theorem demonstrates a more subtle relation.

\begin{thm}[{\cite[Corollary 7.7]{carlson} and \cite[Theorem 3.1]{Carlson93}}]\label{T: dimension}
Suppose that $E$ has rank $k$, $M$ is an $\F E$-module, $r,d$ are the dimension and degree of $V^\#_E(M)$ respectively. We have $p^{k-r}\mid \dim_\F M$ and $dp^{k-r}\leq \dim_\F M$.
\end{thm}

Following \cite[Appendix B, Problem 16]{benson2016}, we define the following.

\begin{defn} Let $V$ be a homogeneous algebraic variety in $\A^k(\F)$. Define \[d_V=\min\{\dim_\F M:V^\#_E(M)\cong V\}\] where $E$ is an elementary abelian $p$-group of rank $k$ and $M$ is finite-dimensional over $\F$.
\end{defn}

The set in the definition is not empty by \cite{Carlson84} and Theorem \ref{T: dimension} gives a lower bound for $d_V$. {In the rank 1 case, the description of $d_V$ is simple. If $\dim V=0$ then $d_V=p$. If $\dim V=1$ then $d_V=1$. To address the rank 2 case, we need the following lemma. }

\begin{lem}\label{L: linear space} Let $E$ be an elementary abelian $p$-group of rank $k$ and $W$ be a linear subspace of $\F^k$ of dimension $r$. There is an $\F E$-module $M$ of dimension $p^{k-r}$ such that $V^\#_E(M)=W$. In particular, we have $d_W=p^{k-r}$.
\end{lem}
\begin{proof} The existence of such module has been shown in \cite[Example 12.1.2]{benson2016}. As such, $d_W$ is bounded above by $p^{k-r}$. Together with Theorem \ref{T: dimension}, we have the equality.
\end{proof}

As such, we get the following result for $d_V$ in the rank 2 case.

{\begin{thm} Suppose that $k=2$. Let $V$ be a homogeneous algebraic variety in $\A^2(\F)$ of dimension $r$. Then \[d_V=\left \{\begin{array}{ll}1&\text{if $r=2$,}\\ dp&\text{if $r=1$ and $d=\deg(\overline{V})$,}\\ p^2&\text{if $r=0$.}\end{array}\right .\]
\end{thm}
\begin{proof} Suppose that $r=2$. We have $V=\A^2(\F)$ and $d_V=1$. Suppose that $r=1$ and $d=\deg(\overline{V})$. We have that $\overline{V}$ is a union of $d$ distinct points $v_1,\ldots,v_d$. For each $1\leq i\leq d$, by Lemma \ref{L: linear space}, there is an $\F E$-module $M_i$ of dimension $p$ such that $V^\#_E(M_i)$ is the line passing through $v_i$. By Theorem \ref{T: basic rank}(iii), $V^\#_E(\bigoplus_{i=1}^d M_i)=V$. As such, $d_V$ is bounded above by $dp$. Using Theorem \ref{T: dimension}, we have $d_V=dp$. Suppose now that $r=0$. We get $V^\#_E(\F E)=\{0\}=V$. Therefore $d_V=p^2$.
\end{proof}}

Let $G$ be a finite group, $E$ be an elementary abelian $p$-subgroup of $G$ and $M$ be an $\F G$-module. There is a natural action of $\Nom_G(E)$ on $V^{\#}_E(M)$ given by the following. If $n \in \Nom_G(E)$ such that $ng_in^{-1} = \prod_{j=1}^k g_j^{a_{ij}}$ for each $i$, and $\alpha = (\alpha_1,\dotsc, \alpha_k) \in \A^k(\F)$, then \[n \cdot \alpha = \left (\sum_{j=1}^k a_{1j}\alpha_j, \dotsc, \sum_{j=1}^k a_{kj}\alpha_j\right ).\]

In this paper, we are interested in the case when $G=\sym{kp}$ and $E_k=\langle g_1,\ldots,g_k\rangle$ where $g_i$'s are the $p$-cycles as in \S\ref{SS: sym}. In this case, the action on the rank variety can be translated as follows.

\begin{lem}\label{L: symmetry} Let $M$ be an $\F\sym{kp}$-module.  We have $\mathrm{N}_{\sym{kp}}(E_k)/\mathrm{C}_{\sym{kp}}(E_k)\cong \F_p^\times\wr \sym{k}$ and, for $\gamma\in \F_p^\times$ in the $i$th component, $\sigma\in\sym{k}$ and $\alpha\in V^\#_{E_k}(M)$,
\begin{align*}
  \gamma\cdot \alpha&=(\alpha_1,\ldots,\gamma\alpha_i,\ldots,\alpha_k),\\
  \sigma\cdot \alpha&=(\alpha_{\sigma^{-1}(1)},\alpha_{\sigma^{-1}(2)},\ldots,\alpha_{\sigma^{-1}(k)}).
\end{align*}
\end{lem}

\subsection{Maximal and generic Jordan types}
The representations of the cyclic group of order $p$, {denoted by $C_p$}, is well-known. Up to isomorphism, for $\F C_p$, there is a unique simple module and the indecomposable modules are $J_1,J_2,\ldots,J_p$ where $\dim_\F J_i=i$, $J_i$'s are uniserial and $J_p$ is projective.  By abuse of notation, we also write $J_i$ for the Jordan block over $\F$ of size $i\times i$ with eigenvalue 0. As such, for $1\leq r\leq i\leq p$, we have
\[\rank(J_i^r)=i-r. \]

Let $E$ be an elementary abelian $p$-group of rank $k$,  $M$ be an $\F E$-module and $B$ be a basis for $M$. For each element $x\in J=\rad(\F E)$, since the matrix representation $[x]_B$ of $x$ with respect to $B$ is nilpotent, we have $[x]_B$ is similar to a diagonal sum of Jordan blocks. Suppose that, for each $1\leq i\leq p$, $J_i$ appears with multiplicity $a_i\in\N_0$. In this case, we say that $x$ has Jordan type \[[p]^{a_p}[p-1]^{a_{p-1}}\cdots[1]^{a_1}.\] The dominance order $\unrhd$ of partitions gives rise to a partial order on the Jordan types, that is $[p]^{a_p}\cdots[1]^{a_1}\succeq [p]^{b_p}\cdots[1]^{b_1}$ if and only if $(p^{a_p},\dots,1^{a_1})\unrhd (p^{b_p},\dots,1^{b_1})$ as partitions.

Let $x\in J$. We say that $x$ has maximal Jordan type on $M$ if the Jordan type of $[x]_B$ is maximal, with respect to the dominance order, among all the matrices of elements in $J$. For the notion of the generic Jordan type of $M$, we refer the reader to \cite[Chapter 3]{benson2016}. Roughly speaking, the generic
Jordan type of $M$ is the Jordan type of $u_\alpha-1$ on $M$ for almost all points $\alpha\in \A^k(\F)$.

Following the work \cite{FPS}, we may regard the set of elements with maximal Jordan type {(denoted as $U_{\max}(M)$ in \cite{benson2016})} as \[\U{E}{M}\subseteq (J/J^2)\backslash\{0\}\] and, as such, subset of $\A^k(\F)$. By slight abuse of notation, for an $\F G$-module $M$ and an elementary abelian $p$-subgroup $E$ of $G$, we also write $\U{E}{M}$ for $\U{E}{\res{M}{E}}$. For our computation purpose in this paper, we require the connection between the maximal and generic Jordan types as below.

\begin{thm}[{\cite[\S4]{FPS}}]\label{T: generic eq maximal} Let $E$ be an elementary abelian $p$-group and $M$ be an $\F E$-module.  All elements with maximal Jordan type on $M$ have the same Jordan type and it is the same as the generic Jordan type of $M$.
\end{thm}

As a consequence, we have the following corollary.

\begin{cor}\label{C: complement} Suppose that the elementary abelian $p$-group $E$ has rank $k$ and the generic Jordan type of $M$ is free. We have $V^\#_E(M)=\U{E}{M}^c$ which is the complement of $\U{E}{M}$ in $\A^k(\F)$.
\end{cor}
\begin{proof} By Theorem \ref{T: generic eq maximal}, the maximal Jordan type of $M$ is $[p]^m$ for some $m$. For any $\alpha\in \A^k(\F)$, $\res{M}{\langle u_\alpha\rangle}$ is not free if and only if $u_\alpha-1$ has Jordan type strictly less than $[p]^m$ in the dominance order. Therefore, $\alpha\in V^\#_E(M)$ if and only if $\alpha\not\in\U{E}{M}$.
\end{proof}

To conclude our preliminary, we would like to record the following result which is a special case of  \cite[Corollary~4.6.2]{benson2016} as the $r$th exterior power coincides with the Schur functor labelled by $(1^r)$.

\begin{prop}\label{P: Umax}
If $r<p$, the generic Jordan type of $\bigwedge^r(M)$ is $\bigwedge^r$ of the generic Jordan type of $M$, and $\U{E}{\bigwedge^r(M)}\supseteq\U{E}{M}$.
\end{prop}

\section{A result on the Green vertex}\label{S: Green Vertex}

In this section, we prove a result concerning the Green vertices. In the literature, a sufficient condition for a Green vertex $Q$ of an indecomposable $\F G$-module $M$ to contain an elementary abelian $p$-subgroup $E$ of rank $k$ is $V^\#_E(M)=\F^k$. Our theorem (see Theorem \ref{T: Green}) generalises such result.

We begin with a lemma.

\begin{lem}\label{L: rkinduction}
Let $E=\langle g_1,g_2,\dots,g_k\rangle$ be an elementary abelian $p$-group of rank $k$, $E'=\langle g_1^{w_{11}}\cdots g_k^{w_{1k}},\dots,g_1^{w_{s1}}\cdots g_k^{w_{sk}}\rangle$ be a subgroup of $E$ of rank $s$ where $0\leq w_{ij}\leq p-1$ for all admissible $i,j$ and $M$ be an $\F E'$-module. Then $V_E^\#(\ind{M}{E})$ lies in the subspace of $\A^k(\F)$ given by \[W=\text{span}\{ (w_{11},\dots ,w_{1k}),(w_{21},\dots ,w_{2k}),\dots, (w_{s1},\dots ,w_{sk})\}.\] In particular, if $E'$ acts trivially on $M$, then $V^\#_{E'}(M)=W$.
\end{lem}
\begin{proof} Since $E'$ is a $p$-group, $M$ has composition factors copies of $\F_{E'}$. Therefore, $\ind{M}{E}$ has a filtration with factors copies of $\ind{\F_{E'}}{E}$. By Lemma \ref{L: filtra}, we have $V^\#_E(\ind{M}{E})\subseteq V^\#_E(\ind{\F_{E'}}{E})$. As such, we only need to prove that $V^\#_E(\ind{\F_{E'}}{E})=W$. %Since $W$ is irreducible (as it is a linear subspace) and $\dim V^\#_E(\ind{\F_{E'}}{E})=\dim V^\#_{E'}(\F)=s=\dim W$ by Theorem \ref{T: basic rank}(v), we only need to prove that $V^\#_E(\ind{\F_{E'}}{E})\subseteq W$.

For $1\leq i\leq s$, let $w_i=(w_{i1},\dots ,w_{ik})$, $f_i=g_1^{w_{i1}}\cdots g_k^{w_{ik}}$ and $\{e_i:1\leq i\leq k\}$ be the standard basis for $\A^k(\F)$.  Let $J=\rad(\F E)$ and $E''=\langle g_{j_1},\ldots,g_{j_t}\rangle$ be the subgroup of $E$ such that $s+t=k$, $1\leq j_1<\cdots<j_t\leq k$ and $E=E'\times E''$. Let $0\neq \alpha\in \A^k(\F)$ and suppose that
\begin{align*}\label{Eq: alpha}
\alpha=\sum_{i=1}^s a_iw_i+\sum_{i=1}^t b_ie_{j_i}=\sum_{i=1}^sa_i\left (\sum_{j=1}^k w_{ij}e_j\right )+\sum_{i=1}^tb_ie_{j_i}
\end{align*} where $a_i,b_i\in\F$. For any $\beta_1,\dots,\beta_k\in \F_p$, we have
\[\prod_{i=1}^k g_i^{\beta_i}=\prod_{i=1}^k (X_i+1)^{\beta_i}
                                          \equiv \sum_{i=1}^k \beta_iX_i +1 (\mod J^2).
\]
Consequently,
\begin{align*}
u_\alpha-1&=\sum_{i=1}^sa_i\left (\sum_{j=1}^k w_{ij}X_j\right )+\sum_{i=1}^t b_iX_{j_i} \nonumber\\
                 &\equiv\sum_{i=1}^sa_i\left (\prod_{j=1}^k g_j^{w_{ij}}-1\right )+\sum_{i=1}^tb_i(g_{j_i}-1))(\mod J^2)\nonumber\\
                 &=\sum_{i=1}^sa_i(f_i-1)+\sum_{i=1}^tb_i(g_{j_i}-1).  \label{Eq: 1}
\end{align*}
Let $z:=\sum_{i=1}^tb_i(g_{j_i}-1)$. Notice $z=0$ if and only if $\alpha\in W$. Also, $N:=\ind{\F_{E'}}{E}\cong \F E''$ where $E'$ acts trivially and $E''$ acts regularly. Therefore, $u_\alpha-1$ acts as $z$ on $N$. As such, $z\neq 0$ if and only if $u_\alpha$ acts freely on $N$. So $V^\#_E(N)=W$.
\end{proof}

In view of Lemma \ref{L: rkinduction}, we have the following definition.

{\begin{defn} For the vector space $V=\F^r$ over $\F$, a subspace $W$ of $V$ is called a base subspace if $W$ has a basis consisting of vectors of the form $(\lambda_1,\ldots,\lambda_r)\in \F_p^r$ where $\F_p$ is the base subfield of $\F$.
\end{defn}}

{For example, $\F^r$ is a base subspace of $\F^r$ with the standard basis. The total number of base subspaces of $\F^r$ is obviously finite.}

\begin{thm}\label{T: Green} Let $M$ be an indecomposable $\F G$-module and $E$ be an elementary abelian $p$-subgroup of $G$ of rank $r$. If $V^\#_E(\res{M}{E})$ is not contained in the union of proper base subspaces of $\F^r$ then $E$ is a subgroup of a Green vertex of $M$.
\end{thm}
\begin{proof} Let $Q$ be a vertex and $S$ be an $\F Q$-source of $M$. By Mackey's formula, we have \[\res{M}{E}\left | \bigoplus \ind{\res{{}^gS}{E\cap {}^gQ}}{E}\right ..\] Consequently, $V^\#_E(\res{M}{E})\subseteq \bigcup V^\#_E(\ind{\res{{}^gS}{E\cap {}^gQ}}{E})$ by Theorem \ref{T: basic rank}(iii). If $E$ is not contained in a conjugate of $Q$, then $E\cap {}^gQ$ is a proper subgroup of $E$. Hence, by Lemma \ref{L: rkinduction}, $V^\#_E(\ind{\res{{}^gS}{E\cap {}^gQ}}{E})$ is contained in a proper base subspace.
\end{proof}

\begin{eg} Let $p=3$, $E=\langle g_1,g_2\rangle\cong C_3\times C_3$ and consider the $\F E$-module $M_{\lambda,\mu}$ given in {\cite[Example 1.13.1]{benson2016}} where
\begin{align*}
X_1&\mapsto \begin{pmatrix}
0&0&0\\ 1&0&0\\ 0&1&0
\end{pmatrix},&
X_2&\mapsto \begin{pmatrix}
0&0&0\\ \lambda&0&0\\ \mu&\lambda&0
\end{pmatrix}.
\end{align*} For $(\alpha_1,\alpha_2)\in\A^2(\F)$, we have \[X_\alpha\mapsto \begin{pmatrix}
0&0&0\\ \alpha_1+\alpha_2\lambda&0&0\\ \alpha_2\mu&\alpha_1+\alpha_2\lambda&0
\end{pmatrix}=:U_\alpha.\] Clearly, $\alpha\in V^\#_E(M_{\lambda,\mu})$ if and only if $\alpha_1+\alpha_2\lambda=0$, i.e., $V^\#_E(M_{\lambda,\mu})$ is the line passing through $(-\lambda,1)$. If $\lambda\not\in \F_3$, then $(-\lambda,1)$ is a point of the rank variety but not belonging to the union of the base lines. In this case, by Theorem \ref{T: Green}, $E$ is the Green vertex of $M_{\lambda,\mu}$.

The converse is however incorrect even for this example. Let $\lambda=0$ and $\mu\neq 0$. {We claim that $E$ is the Green vertex of $M$ but $V^\#_E(M_{0,\mu})$ is contained in the union of some proper base lines.}
It is clear that $V^\#_E(M_{0,\mu})$ is the line passing through $(0,1)$.  Let $E'$ be a Green vertex of $M_{0,\mu}$.  Since $M_{0,\mu}$ is not projective, $E'>\{1\}$. Suppose that $E'\neq E$, i.e., $E'=\langle g_1^ag_2^b\rangle$ for some $0\leq a,b\leq 2$.  Since $M_{0,\mu}\mid \ind{\res{M_{0,\mu}}{E'}}{E}$, by Lemma \ref{L: rkinduction}, we have \[V^\#_E(M_{0,\mu})\subseteq V^\#_E(\ind{\res{M_{0,\mu}}{E'}}{E})\subseteq  \mathrm{span}\{(a,b)\}\] and hence $a=0$ and, without loss of generality, $b=1$. However, $\res{M_{0,\mu}}{\langle g_2\rangle}=U_1\oplus U_2$ where $\dim_\F U_i=i$ and $U_i$'s are indecomposable. By Green's indecomposability theorem (see \cite[Theorem 3.13.3]{Benson1}) and Krull-Schmidt Theorem, for \[M_{0,\mu}\mid \ind{\res{M_{0,\mu}}{E'}}{E}\cong \ind{U_1}{E}\oplus \ind{U_2}{E},\] we must have $\ind{U_1}{E}\cong M_{0,\mu}$. However, as $U_1$ is the trivial $\F E'$-module, \[\res{\ind{U_1}{E}}{E'}\cong \F\oplus\F\oplus \F \not\cong U_1\oplus U_2\cong \res{M_{0,\mu}}{E'}.\]
\end{eg}

\section{The natural simple module $D(1)$}\label{S: D1}
Let $n=kp>2$ and we consider the natural simple module $D(1)=D^{(kp-1,1)}$.  As in \S\ref{SS: sym}, let $E_k$ be the elementary abelian $p$-subgroup of $\sym{kp}$ generated by the $k$ disjoint $p$-cycles $g_1,\ldots,g_k$.  Our main result describes the maximal (or generic) Jordan type of $\res{D(1)}{E_k}$ for $p\geq 3$ {and the corresponding maximal Jordan set}. We note that the $p=2$ case has been dealt with in \cite{Jiang21}. Although our method in this section works even in this case, we have decided to leave it out so that the presentation would be neater. We assume $p\geq 3$ henceforth.

To begin, we recall the $(kp-1,1)$-tabloids $\mathfrak{t}_i$, $e_i=\mathfrak{t}_i-\mathfrak{t}_1$ (for $i\neq 1$) and the basis $\{ \e i:  i=3,\ldots,kp\}$ we have described for $D(1)$ where $\e i=e_i+D^{(n)}$ and $D^{(n)}$ is identified as the trivial submodule of $S^{(n-1,1)}$ spanned by $\sum_{i= 2}^{kp}e_i$. Similarly, we write $\t i$ for $\mathfrak{t}_i+D^{(n)}$.  The basis is not particularly convenient for our computation. As such, our first step is to pick another basis for $D(1)$ which interacts well with the action of $E_k$. For $2\leq i\leq k$ and $1\leq r\leq p$, let $e_{i,r}=e_{(i-1)p+r}$ and $\mathfrak{t}_{i,r}=\mathfrak{t}_{(i-1)p+r}$. Thus, $\e {i,r}=\t {i,r}-\t 1$ and $g_i\t {i,r}=\t {i,g_1r}$. Recall that, for each $i=1,\ldots,k$, we have $X_i=g_i-1$.
%Define $b_1=\e 3 -\e 4$ and, for $2\leq i \leq k$, define $b_i=\e {i,1} -\e 3$.

%Let \[B_1=\{ b_1,X_1^1b_1,\ldots,X_1^{p-3}b_1\},\] and for $2\leq i\leq k$, let \[B_i=\{ b_i,X_i^1b_i,\ldots,X_1^{p-1}b_i\}.\]

\begin{defn} 
Let $b_1=\e 3$ if $p=3$ and $b_1=\e 3 -\e 4$ if $p\geq 5$, $\B_1=\{ b_1,X_1^1b_1,\ldots,X_1^{p-3}b_1\}$. For $2\leq i\leq k$, let $b_i=\e {i,1} -\e 3$ and $\B_i=\{ b_i,X_i^1b_i,\ldots,X_i^{p-1}b_i\}$.
Let \[\B=\bigcup_{i=1}^{k} \B_i\].
\end{defn}

The next two lemmas demonstrate the structure of $\res{D(1)}{E_k}$.

\begin{lem}\label{L: D(1) basis} The set $\B$ is a basis for $D(1)$. Moreover,
\begin{enumerate}[(i)]
\item for $1\leq r \leq p-1$ and $2\leq i\leq k$, we have
\[ X_i^rb_i=\sum_{s=1}^{r+1} (-1)^{r-s+1} \binom{r}{s-1} \e {i,s};\]
\item if $p\geq 5$, we have for $1\leq r\leq p-4$, we have
\begin{align*}
X_1^rb_1&=\sum_{s=1}^{r+2} (-1)^{r-s+3} \binom{r+1}{s-1} \e {s+2},\\
X_1^{p-3}b_1&=\sum_{s=1}^{p-2}s\e {s+2}.
\end{align*}
\end{enumerate}
\end{lem}
\begin{proof} We show the equations in the statement by induction on $r$. Suppose that $2\leq i\leq k$. Notice that \[b_i=\e {i,1} -\e 3=(\t {i,1}-\t 1)-(\t 3-\t 1)=\t {i,1}-\t 3,\] and we also have
\begin{align*}
 \sum_{s=1}^{r+1} (-1)^{r-s+1} \binom{r}{s-1} \e {i,s}
 &=\sum_{s=1}^{r+1} (-1)^{r-s+1} \binom{r}{s-1}(\t {i,s}-\t 1)\\
 &=\left(\sum_{s=1}^{r+1} (-1)^{r-s+1} \binom{r}{s-1}\t {i,s}\right )-\left (\sum_{s=1}^{r+1} (-1)^{r-s+1} \binom{r}{s-1}\right )\t 1\\
 &=\sum_{s=1}^{r+1} (-1)^{r-s+1} \binom{r}{s-1}\t {i,s}.
 \end{align*}
When $r=1$, we have
\begin{align*}
X_ib_i&=(g_i-1)(\t {i,1}-\t 3) =(g_i\t {i,1}-g_i\t 3)-(\t {i,1}-\t 3)=\t {i,g_11}-\t 3-\t {i,1}+\t 3\\
           &=\t {i,2}-\t {i,1}=-\e {i,1}+\e {i,2}.
\end{align*}
Suppose that part (i) holds for some $1\leq r\leq p-2$. We have
\begin{align*}
X_i(X_i^rb_i)&=(g_i-1)\left (\sum_{s=1}^{r+1} (-1)^{r-s+1} \binom{r}{s-1} \t {i,s}\right )\\
                    &=\left (\sum_{s=1}^{r+1} (-1)^{r-s+1} \binom{r}{s-1} g_i\t {i,s}\right )-\left (\sum_{s=1}^{r+1} (-1)^{r-s+1} \binom{r}{s-1} \t {i,s}\right )\\
                    &=\sum_{s=1}^{r+1} (-1)^{r-s+1} \binom{r}{s-1} \t {i,s+1}+\sum_{s=1}^{r+1} (-1)^{r+1-s+1} \binom{r}{s-1} \t {i,s}\\
                    &=\sum_{s=2}^{r+2} (-1)^{r-s+2} \binom{r}{s-2} \t {i,s}+\sum_{s=1}^{r+1} (-1)^{r+1-s+1} \binom{r}{s-1} \t {i,s}\\
                    &=(-1)^{r+1}\t {i,1}+\t {i,r+2}+\sum_{s=2}^{r+1} (-1)^{r-s+2} \left (\binom{r}{s-2}+\binom{r}{s-1}\right ) \t {i,s}\\
                    &=\sum_{s=1}^{r+2} (-1)^{r-s+2} \binom{r+1}{s-1} \t {i,s}.
\end{align*}
For part (ii), suppose $p\geq 5$. If $r=1$, then
\begin{align*}
X_1b_1&=(g_1-1)(\e 3-\e 4)=(g_1-1)(\t 3-\t 4)=(\t 4-\t 5)-(\t 3-\t 4)\\
            &=-\t 3+2\t 4-\t 5=-\e 3+2\e 4-\e 5.
\end{align*}
For $1\leq r\leq p-5$, we have
\begin{align*}
X_1(X_1^rb_1)&=(g_1-1)\left (\sum_{s=1}^{r+2} (-1)^{r-s+3} \binom{r+1}{s-1} \t {s+2}\right )\\
                       &=\sum_{s=1}^{r+2} (-1)^{r-s+3} \binom{r+1}{s-1} g_1\t {s+2}-\sum_{s=1}^{r+2} (-1)^{r-s+3} \binom{r+1}{s-1} \t {s+2}\\
                       &=\sum_{s=1}^{r+2} (-1)^{r-s+3} \binom{r+1}{s-1}\t {s+3}+\sum_{s=1}^{r+2} (-1)^{r+1-s+3} \binom{r+1}{s-1} \t {s+2}\\
                       &=\sum_{s=2}^{r+3} (-1)^{r+1-s+3} \binom{r+1}{s-2}\t {s+2}+\sum_{s=1}^{r+2} (-1)^{r+1-s+3} \binom{r+1}{s-1} \t {s+2}\\
                       &=(-1)^{r+3}\t {3}+(-1)\t {r+5}+\sum_{s=2}^{r+2} (-1)^{r+1-s+3} \left (\binom{r+1}{s-2}+\binom{r+1}{s-1}\right ) \t {s+2}\\
                       %&=(-1)^{r+3}\t {3}+(-1)\t {r+5}+\sum_{s=2}^{r+2} (-1)^{r+1-s+3} \binom{r+2}{s} \t {s+2}\\
                       &=\sum_{s=1}^{r+3} (-1)^{r+1-s+3} \binom{r+2}{s-1} \e {s+2}.
\end{align*}
Lastly, when $r=p-4$, we have
\begin{align*}
X_1(X_1^{p-4}b_1)&=(g_1-1)\left (\sum_{s=1}^{p-2} (-1)^{p-1-s} \binom{p-3}{s-1} \t {s+2}\right )\\
&=\left (\sum_{s=1}^{p-2} (-1)^{s+1} \binom{p-2}{s-1} \t {s+2}\right )-g_1\t p\\
                              &=\left (\sum_{s=1}^{p-2} s \t {s+2}\right )- \t 1
                              %&=(\sum_{s=1}^{p-2} s \t {s+2})+\sum_{s=2}^{kp}\t s\\
                       % =\t 3+2\t 4+\ldots+(p-2)\t p+\sum_{s=2}^{kp}\t s\\
                              =\t 2+2\t 3+\ldots+(p-1)\t p+\sum_{s=p+1}^{kp}\t s\\
                     %         &=\e 2+2\e 3+\ldots+(p-1)\e p+\sum_{s=p+1}^{kp}\e s\\
                              &=\e 3+2\e 4+\ldots+(p-2)\e p.
\end{align*}
{Given the description above, it is clear that, for $2\leq i\leq k$, $B_i$ is linearly independent and $\mathrm{span}(B_i)= \mathrm{span}\{\e{i,1},\dots,\e{i,p}\}$ as vector spaces. Furthermore, $\bigcup_{i=1}^{p-2} B_i$ is a basis if and only if the set $\{b_1,\dots, X_1^{p-4}b_1,\sum_{i=3}^{p}(i-2)\e {i}\}$ is linearly independent. We have $(b_1,\dots, X_1^{p-4}b_1,\sum_{i=3}^{p}(i-2)\e {i}\})^\top=A(\e 3,\e 4,\dots, \e p)^\top$ where $\top$ denotes the transpose of a matrix and $A$ is the matrix where, for $1\leq i,j\leq p-2$,
\[A_{ij}=(-1)^{i-j+2}\binom{i}{j-1}.
\]
Thus, the set $\{b_1,\dots, X_1^{p-4}b_1,\sum_{i=3}^{p}(i-2)\e {i}\}$ is linearly independent if and only if $\rank A=p-2$. For $1\leq i\leq p-3$, let $T_i$ represent the elementary row operation adding $\sum_{k=1}^{i}(-1)^{i-k+2}\binom{i+1}{k}R_k$ to the $(i+1)$th row where $R_k$ denotes the $k$th row of $R$.
Performing the row operations in the order $T_1, T_2, \dots, T_{p-3}$ on $A$, we obtain the following matrix $A'$:
\[A'=\begin{pmatrix}
1&-1&\\
1&0&-1&\\
\vdots&\vdots&\ddots&\ddots&\\
1&0&\cdots&0&-1\\
1&0&\cdots&0&0
\end{pmatrix}.\]
Thus, $\rank A=\rank {A'}=p-2$. We conclude that $\B$ is a basis for $D(1)$.
}
\end{proof}

\begin{lem}\label{L: action}
Let $1\leq j\leq k$. We have
\begin{enumerate}[(i)]
\item for $2\leq i\leq k$ and $0\leq r\leq p-1$,
%$X_i(X_i^rb_i)=X_i^{r+1}b_i$, $X_{i'}(X_i^rb_i)=0$ and $X_1(b_i)=b1$
                  %\begin{align*} X_i(X_i^rb_i)&=X_i^{r+1}b_i,\\ X_{i'}(X_i^rb_i)&=0,\\ X_1(b_i)&=b1.\end{align*}
                  \[
                   X_j(X_i^rb_i)=\left \{\begin{array}{ll}
                   %X_i^{r+1}b_i &\text{if $j=i$ and $r\neq p-1$,}\\
                   b_1&\text{if $j=1$ and $r=0$,}\\
                   0 &\text{if $j\neq i$ or $r=p-1$.}\end{array}\right .
                   \]
\item for $0\leq r\leq p-3$,
      \[
         X_j(X_1^rb_1)=\left \{\begin{array}{ll}
         %X_1^{r+1}b_1 &\text{if $j=1$ and $r\neq p-3$,}\\
          \sum_{i=2}^{k} X_i^{p-1}b_i &\text{if $j=1$ and $r=p-3$,}\\
          0 &\text{if $j\neq 1$.}\end{array}\right .
         \]
\end{enumerate} In particular, $X_iX_j$ annihilates $D(1)$ for all $1\leq i\neq j\leq k$.
\end{lem}
\begin{proof}
{Recall that $b_1= \t 3-\t 1$ if $p=3$ and $\t 3-\t 4$ if $p\geq 5$, and $b_i=\t {i,1}-\t 3$ for $2\leq i\leq k$.} Suppose $2\leq i\neq j\leq k$. If $p=3$, we have
\begin{align*}
X_1b_i&=(g_1-1)(\t {i,1} -\t 3)=(\t {i,1} -\t 1)-(\t {i,1}-\t {3})=\t 3-\t 1=b_1.
\end{align*}
If $p\geq 5$, we have
\[
X_1b_i=(g_1-1)(\t {i,1} -\t 3)=(\t {i,1} -\t 4)-(\t {i,1}-\t {3})=\e 3-\e 4=b_1.\]
For $2\leq i \leq k$, we have
\[X_i(X_i^{p-1}b_i)=X_i^{p}b_i=0b_i=0.
\]
Since $g_{j}(\t {i,r})=\t {i,r}$, we have that $X_{j}(X_i^rb_i)=0$. Thus, the equalities in part (i) hold. For part (ii), we only show $X_1(X_1^{p-3}b_1)=\sum_{i=2}^k X_i^{p-1}b_i$ and the rest are very similar to the cases in part (i). If $p=3$, we have
\begin{align*}     
X_1 b_1&=(g_1-1)(\t 3-\t 1)\\
             &=(\t 1-\t 2)-(\t 3-\t 1)\\
             &=-\e 2-\e 3\\
             &=\sum_{s=4}^{3k} \e s=\sum_{i=2}^k X_i^{2}b_i.
\end{align*}
If $p\geq 5$, we have
\begin{align*}
X_1(X_1^{p-3}b_1)&=(g_1-1)\left (\t 2+2\t 3+\ldots+(p-1)\t p\right )\\
                              &=\t 3+2\t 4+\ldots+(p-2)\t p-\t 1-(\t 2+2\t 3+\ldots+(p-1)\t p)\\
                              &=-\t 2-\t 3-\ldots -\t p +\left (\sum_{s=2}^{kp}\t s\right )\\
                              &=\sum_{s=p+1}^{kp}\t s=\sum_{i=2}^k X_i^{p-1}b_i.
\end{align*}

\end{proof}

The following diagram depicts the actions of the $X_j$'s on the basis $\B$ of $D(1)$. The blue arrows represent the action of $X_1$ and the remaining arrows represent the respective actions of $X_j$'s for $j\neq 1$. The $+$ in the diagram indicates that $X_1^{p-2}b_1=\sum^k_{j=2}X_j^{p-1}b_j$.

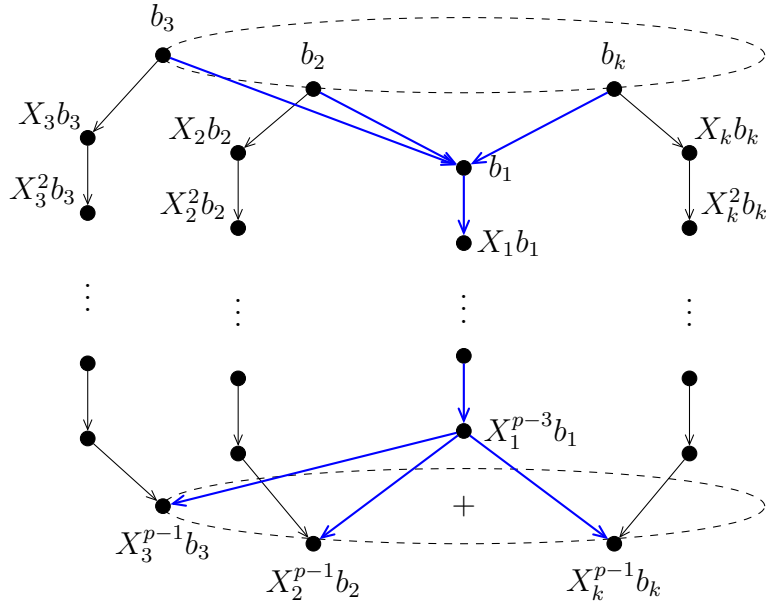
\[\begin{tikzpicture}
  \draw[dashed] (0,0) ellipse (4cm and .5cm);
  \draw[dashed] (0,-6) ellipse (4cm and .5cm);
  \node (b2) at (-4,0) [circle,draw,fill=black,scale=.5] {};
  \node (b21) at (-5,-1.1) [circle,draw,fill=black,scale=.5] {};
  \node (b22) at (-5,-2.1) [circle,draw,fill=black,scale=.5] {};
  \node at (-5,-3.1) {$\vdots$};
  \node (b23) at (-5,-4.1) [circle,draw,fill=black,scale=.5] {};
  \node (b24) at (-5,-5.1) [circle,draw,fill=black,scale=.5] {};
  \node (b25) at (-4,-6) [circle,draw,fill=black,scale=.5] {};
  \node (b3) at (-2,-0.45) [circle,draw,fill=black,scale=.5] {};
  \node (b31) at (-3,-1.3) [circle,draw,fill=black,scale=.5] {};
  \node (b32) at (-3,-2.3) [circle,draw,fill=black,scale=.5] {};
  \node at (-3,-3.3) {$\vdots$};
  \node (b33) at (-3,-4.3) [circle,draw,fill=black,scale=.5] {};
  \node (b34) at (-3,-5.3) [circle,draw,fill=black,scale=.5] {};
  \node (b35) at (-2,-6.5) [circle,draw,fill=black,scale=.5] {};
  \node (b1) at (0,-1.5) [circle,draw,fill=black,scale=.5] {};
  \node (b11) at (0,-2.5) [circle,draw,fill=black,scale=.5] {};
  \node at (0,-3.25) {$\vdots$};
  \node (b13) at (0,-4) [circle,draw,fill=black,scale=.5] {};
  \node (b14) at (0,-5) [circle,draw,fill=black,scale=.5] {};
  \node (b4) at (2,-0.45) [circle,draw,fill=black,scale=.5] {};
  \node (b41) at (3,-1.3) [circle,draw,fill=black,scale=.5] {};
  \node (b42) at (3,-2.3) [circle,draw,fill=black,scale=.5] {};
  \node at (3,-3.3) {$\vdots$};
  \node (b43) at (3,-4.3) [circle,draw,fill=black,scale=.5] {};
  \node (b44) at (3,-5.3) [circle,draw,fill=black,scale=.5] {};
  \node (b45) at (2,-6.5) [circle,draw,fill=black,scale=.5] {};
  \draw[-angle 45,thick,color=blue]  (b3) -- (b1);
  \draw[-angle 45] (b3) -- (b31);
  \draw[-angle 45]  (b31) -- (b32);
  \draw[-angle 45]  (b33) -- (b34);
  \draw[-angle 45] (b34) -- (b35);
  \draw[-angle 45,thick,color=blue] (b2) -- (b1);
  \draw[-angle 45] (b2) -- (b21);
  \draw[-angle 45] (b21) -- (b22);
  \draw[-angle 45] (b23) -- (b24);
  \draw[-angle 45] (b24) -- (b25);
  \draw[-angle 45,thick,color=blue] (b1) -- (b11);
  \draw[-angle 45,thick,color=blue] (b13) -- (b14);
  \draw[-angle 45,thick,color=blue] (b14) -- (b25);
  \draw[-angle 45,thick,color=blue] (b14) -- (b35);
  \draw[-angle 45,thick,color=blue] (b14) -- (b45);
  \draw[-angle 45,thick,color=blue] (b4) -- (b1);
  \draw[-angle 45] (b4) -- (b41);
  \draw[-angle 45] (b41) -- (b42);
  \draw[-angle 45] (b43) -- (b44);
  \draw[-angle 45] (b44) -- (b45);
  \node at (-4,0.5) {$b_3$};
  \node at (-5.5,-0.8) {$X_3b_3$};
  \node at (-5.6,-1.8) {$X_3^2b_3$};
  \node at (-4,-6.5) {$X_3^{p-1}b_3$};
  \node at (-2,0) {$b_2$};
  \node at (-3.5,-1) {$X_2b_2$};
  \node at (-3.6,-2) {$X_2^2b_2$};
  \node at (-2,-7) {$X_2^{p-1}b_2$};
  \node at (2,0) {$b_k$};
  \node at (3.5,-1) {$X_kb_k$};
  \node at (3.6,-2) {$X_k^2b_k$};
  \node at (2,-7) {$X_k^{p-1}b_k$};
  \node at (0.5,-1.5) {$b_1$};
  \node at (0.6,-2.5) {$X_1b_1$};
  \node at (0.9,-5) {$X_1^{p-3}b_1$};
  \node at (0,-6) {$+$};
\end{tikzpicture}\]

For $0\neq \alpha\in \A^k(\F)$, let $X_\alpha=\sum_{i=1}^k \alpha_iX_i$ as in the preliminary.  By Lemma \ref{L: action}, we have \[
\begin{tikzpicture}
\matrix(m)[matrix of math nodes,left delimiter=(,right delimiter=),row sep=0.05cm,column sep=0.05cm]
{
        0     &{}           &{}              &{}             &\alpha_1  &{}        &{} &{}        &{}&{}&\alpha_1  &{}           &{}  &{} \\
\alpha_1 &    0    &{}               &{}             &{}              &{}          &{}            &{}         &\ldots&{}&{}             &{}              &{}            &{}  \\
        {}          &\ddots& \ddots    &{}             &{}     &{}           &{}            &{}         &{}&&{} &{}   &{}             &{}            &{}  \\
          {}        &{}          & \alpha_1&     0    &{}              &{}           &{}           &{}       &{}&{}&{}&{}             &{}             &{}            &{} \\
%&{}&{}&{}&{}&{}&{}&{}&{}&{}&{}&{}&{}&{}&{}&{} \\
        {}    &{}            &{}              &{}           &   0   &{}            &{}            &{}                &{}&{}&{}&{}              &{}            &{}       &{}  \\
         {}       &{}            &{}              &{}             &\alpha_2  &    0     &{}             &{}      &{}&{}&{}&{}                 &{}             &{}       &{}  \\
        {}        &{}            &{}              &{}             &{}                & \ddots &  \ddots&{}      &{}&{}&{} &{}               &{}             &{}         &{}  \\
       {}        &{}            &{}             &\alpha_1&{}               &{}            &\alpha_2&  0  &{}&{}&{} &{}                 &{}             &{}       &{} \\
  {} &\vdots&{}&{}&{}&{}&{}&{}&\ddots&{}&{}&{}&{} \\
    {}      &{}            &{}              &{}           &{}       &{}            &{}            &{}      &{}&{}&      0    &{}            &{}       &{}  \\
          {}   &{}            &{}              &{}            &{}        &{}            &{}             &{}      &{}&{}&   \alpha_k&  0       &{}       &{}  \\
         {}       &{}            &{}              &{}            &{}        &{}            &{}              &{}      &{}&{}&{}             &  \ddots & \ddots&{}  \\
        {}       &{}            &{}             &\alpha_1&{}         &{}            &{}             &{}   &{}&{}&{}                &{}          &\alpha_k  &0 \\
        };
         \node[fit=(m-1-1)(m-4-4),draw=blue,dashed,inner sep=0.5mm]{};
         \node[fit=(m-5-5)(m-8-8),draw=blue,dashed,inner sep=0.5mm]{};
         \node[fit=(m-10-11)(m-13-14),draw=blue,dashed,inner sep=0.5mm]{};
         \draw[dashed,color=orange] (-2,2) rectangle (0.8,4.8);
         \draw[dashed,color=orange] (2,2) rectangle (5,4.8);
         \draw[dashed,color=orange] (-5.1,-0.95) rectangle (-2.2,1.8);
         \draw[dashed,color=orange] (-5.1,-4.75) rectangle (-2.2,-1.9);
         \node at (-2.8,4.5) {$S(1,1)$};
         \node at (0.1,4.5) {$S(1,2)$};
         \node at (4.3,4.5) {$S(1,k)$};
         \node at (-2.9,1.5) {$S(2,1)$};
         \node at (0.2,1.5) {$S(2,2)$};
         \node at (-2.9,-2.2) {$S(k,1)$};
         \node at (4.2,-2.2) {$S(k,k)$};
         \node at (-6.7,0) {$[X_\alpha]_\B=$};
\end{tikzpicture}\] where we have deliberately left out the zero entries and it is comprised of the submatrices $S(i,j)$, where
\begin{enumerate}[(a)]
\item the diagonal submatrices $S(i,i)$ are square of sizes $(p-2)\times (p-2)$ if $i=1$ and $p\times p$ if $i\geq 2$,
\item for $2\leq i\leq k$, $S(i,1)$ is the matrix with the $(p,p-2)$-entry $\alpha_1$ and zero elsewhere,
\item for $2\leq j\leq k$, $S(1,j)$ is the matrix with the $(1,1)$-entry $\alpha_1$ and zero elsewhere, and
\item $S(i,j)$ is the zero matrix otherwise.
\end{enumerate}

Notice that, in the case $k=1$, $[X_\alpha]_\B=S(1,1)$. In order to compute the Jordan type of $[X_\alpha]_\B$, in general, we need to know $\rank([X_\alpha]_\B^r)$ for all $1\leq r\leq p-1$. The next lemma addresses such computation.  Recall the polynomial $p_k$ in Equation \ref{Eq: pk}.

\begin{lem}\label{rank}
Let $k\geq 2$ and $S$ be the matrix $[X_\alpha]_\B$. We have
\begin{enumerate}[(i)]
\item $\rank(S)\leq (k-1)(p-1)+p-3$ and the equality holds if and only if all $\alpha_i$'s are nonzero;
\item if $p\geq 5$ and all $\alpha_i$'s are nonzero, then $\rank(S^{p-3})=3k-2$ and $\rank(S^{p-2})=2k-2$;
\item if all $\alpha_i$'s are nonzero, then $\rank(S^{p-1})\leq k-1$ and the equality holds if and only if $p_k(\alpha_1,\alpha_2,\ldots,\alpha_k)\neq 0$.
\end{enumerate}
\end{lem}

\begin{proof} In this proof, for a matrix $A$, the $j$th row of $A$ is denoted as $A_j$ for all admissible $j$.
Notice that $S$ has at least $k-1$ zero rows. If $\alpha_1=0$, then $\rank(S)\leq (k-1)(p-1)$. Also, if $\alpha_i=0$ for some $i=2,3,\ldots,k$, then $\rank(S)\leq (k-2)(p-1)+p-2$. Suppose $\alpha_i\neq 0$ for $i=1,2,\ldots,k$. We have
\[S_1=\sum_{i=2}^{k} \frac{\alpha_1}{\alpha_i}S_{p(i-1)},\]
and the remaining nonzero rows form a basis for the row space. In this case, $\rank S=kp-2-k$ and the proof for part (i) is complete.

Now we suppose all the $\alpha_i$'s are nonzero. Argue inductively on $m$, we can work out the explicit matrix for $S^m$. We leave it to the reader to observe how $S^m$ look like while we have presented $S$, $S^{p-3}$ and $S^{p-2}$ here. We have
\[
S^{p-3}=\begin{pNiceMatrix}[first-row,first-col,small]
              &1 &2& \ldots &{p-2}&{p-1}&{p}  &{p+1}&\ldots &2p-2&\ldots&{kp-p-1}&{kp-p}
&{kp-p+1}   & \dots&{kp-2}\\
1       &        &       &          &            &             & &           &             &         &         &&                    &
&                     &          &              \\
\vdots    &        &       &          &            &             &&            &             &         &         &&                    &
&                     &          &              \\
{p-3} &  0      &       &          &            &\alpha_1^{p-3}&  &       &      &         &       \ldots  &\alpha_1^{p-3}&
&                     &          &              \\
{p-2} &\alpha_1^{p-3}&0&  &             &    0         &            &     &       &         &         &0&                    &
&                     &          &              \\
\hdottedline
{p-1} &&&  &             &             &            &     &       &         &         &&                    &
&                     &          &              \\
%\hline
\vdots    &         &       &          &            &             &      &      &             &         &        &&                    &
&                     &          &              \\
{2p-4}&         &       &          &            &\alpha_2^{p-3}& &      &       &         &         &&                    &
&                     &          &              \\
{2p-3}&         &       &          &            &        0     &\alpha_2^{p-3}& &         &   &      &&                    &
&                     &          &              \\
{2p-2}&     0    &\alpha_1^{p-3}&&      &       0      &      0      &\alpha_2^{p-3}&&    &   &&                    &
&                     &          &              \\
\hdottedline
%\hline
\vdots&&\vdots&&&&&&&&\ddots&&&&\\
\hdottedline
%\hline
kp-p-1&&&&&&&&&&&&&&\\
\vdots&&&&&&&&&&&&&&\\
{kp-4}&         &       &          &            &             &           &   &             &         &                             &\alpha_k^{p-3}
&                     &          &              \\
{kp-3}&         &       &          &            &             &        &     &             &          &         &     0
&\alpha_k^{p-3}&          &              \\
{kp-2}&      0   &\alpha_1^{p-3}&&      &             &            &     &          &         &         &      0
&              0       &\alpha_k^{p-3}&              \\
\CodeAfter
\tikz \draw[dotted,color=black]  (1-|5) -- (16-|5);
\tikz \draw[dotted,color=black]  (1-|10) -- (16-|10);
\tikz \draw[dotted,color=black]  (1-|11) -- (16-|11);
\end{pNiceMatrix}\] where zero entries are deliberatedly left out. Notice that there are only $2+(k-1)3=3k-1$ nonzero rows with
\[S^{p-3}_{p-3}=\sum_{i=2}^{k} \left (\frac{\alpha_1}{\alpha_i}\right )^{p-3}S^{p-3}_{ip-4}.\]
Thus, $\rank(S^{p-3})=3k-2$. For $S^{p-2}$, we have
\[
S^{p-2}=\begin{pNiceMatrix}[first-row,first-col,small]
              &1 &2& \ldots &{p-2}&{p-1}&{p}  &\ldots &2p-2&\ldots&{kp-p-1}&{kp-p}
  & \dots&{kp-2}\\
1       &        &       &          &            &             & &           &             &         &         &&                    &
&                     &          &              \\
\vdots    &        &       &          &            &             &&            &             &         &         &&                    &
&                     &          &              \\
{p-3} &        &       &          &            &&  &       &      &         &       &&
&                     &          &              \\
{p-2} &&&  &             &    \alpha_1^{p-2}        &            &            &         &    \ldots     & \alpha_1^{p-2}     &                    &
&                     &          &              \\
\hdottedline
{p-1} &&&  &             &             &            &     &       &         &         &&                    &
&                     &          &              \\
%\hline
\vdots    &         &       &          &            &             &      &      &             &         &        &&                    &
&                     &          &              \\
{2p-3}&       &       &          &            &        \alpha_2^{p-2}&& &         &   &      &&                    &
&                     &          &              \\
{2p-2}&     \alpha_1^{p-2}&&&      &       0           &\alpha_2^{p-2}&&  &  &   &&                    &
&                     &          &              \\
\hdottedline
%\hline
\vdots&\vdots&&&&&&&&\ddots&&&&\\
\hdottedline
%\hline
kp-p-1&&&&&&&&&&&&&&\\
\vdots&&&&&&&&&&&&&&\\

{kp-3}&         &       &          &            &             &             &             &          &
&\alpha_k^{p-2}&          &              \\
{kp-2}   &\alpha_1^{p-2}&&      &             &            &     &          &         &
&              0       &\alpha_k^{p-2}&              \\
\CodeAfter  % added <<<<<<<<<<<<
\tikz \draw[dotted,color=black]  (1-|5) -- (16-|5);
\tikz \draw[dotted,color=black]  (1-|9) -- (16-|9);
\tikz \draw[dotted,color=black]  (1-|10) -- (16-|10);
\end{pNiceMatrix}.\] There are only $1+(k-1)2=2k-1$ nonzero rows in $S^{p-2}$ with
\[S^{p-2}_{p-2}=\sum_{i=2}^{k} \left (\frac{\alpha_1}{\alpha_i}\right )^{p-2}S^{p-2}_{ip-3}.\]
Thus, $\rank(S^{p-2})=2k-2$ and the proof for part (ii) is complete.

We now prove part (iii). Using the matrix $S^{p-2}$, we obtain
\[
S^{p-1}=\begin{pNiceMatrix}[first-row,first-col,small]
              &\ldots &{p-1} &\ldots   &{2p-1}&\ldots &{kp-p-1} & \ldots&\\
%\hline
\vdots    &          &             &            &             &          &            & \\
{2p-2}&         &\alpha_1^{p-1}+\alpha_2^{p-1}&&\alpha_1^{p-1}&\ldots&  \alpha_1^{p-1}         &\\
%\hline
\vdots&&&&&&&\\
{3p-2}&        &\alpha_1^{p-1}& &\alpha_1^{p-1}+\alpha_3^{p-1}&\ldots&\alpha_1^{p-1}&\\
%\hline
\vdots&&\vdots&&\vdots&\ddots&\vdots&\\
{kp-2}&          &\alpha_1^{p-1}& &\alpha_1^{p-1} &    \ldots    &\alpha_1^{p-1}+\alpha_k^{p-1}&        \\
\end{pNiceMatrix}\] where we now only highlight the possibly nonzero entries with their corresponding rows and columns. There are only $k-1$ nonzero rows in $S^{p-1}$ and $\rank (S^{p-1})=k-1$ if and only if this following $(k-1)\times (k-1)$-matrix has full rank:
\[B=\begin{pmatrix}
\alpha_1^{p-1}+\alpha_2^{p-1}&\alpha_1^{p-1}                         &\ldots         &\alpha_1^{p-1}\\
\alpha_1^{p-1}                         &\alpha_1^{p-1}+\alpha_3^{p-1}&\ldots         &\alpha_1^{p-1}\\
\vdots &\vdots&\ddots&\vdots\\
\alpha_1^{p-1}                         &\alpha_1^{p-1}                         &\ldots         &\alpha_1^{p-1}+\alpha_k^{p-1}\\
\end{pmatrix}.
\]
We claim that $\det B=p_k(\alpha_1,\ldots,\alpha_k)$ and therefore proving part (iii). Performing row operations, we have
\[\xymatrix{B\ar[r]^-T&{\begin{pmatrix}
\alpha_2^{p-1}&0                         &\ldots         &-\alpha_k^{p-1}\\
0                     &\alpha_3^{p-1}   &\ldots         &-\alpha_k^{p-1}\\
\vdots &\vdots&\ddots&\vdots\\
\alpha_1^{p-1}                         &\alpha_1^{p-1}                         &\ldots         &\alpha_1^{p-1}+\alpha_k^{p-1}
\end{pmatrix}=:B'}\ar[d]^-U\\
&
{B'':=\begin{pmatrix}
\alpha_2^{p-1}&0                         &\ldots         &-\alpha_k^{p-1}\\
0                     &\alpha_3^{p-1}    &\ldots         &-\alpha_k^{p-1}\\
\vdots &\vdots&\ddots&\vdots\\
0                     &0                         &\ldots         &\alpha_1^{p-1}+\alpha_k^{p-1}+\sum_{i=2}^{k-1}\frac{\alpha_1^{p-1}\alpha_k^{p-1}}{\alpha_i^{p-1}}
\end{pmatrix}}}\] where $T$ corresponds to the row operations of subtracting the $i$th row by the $(k-1)$th row, one for each $1\leq i\leq k-2$, and $U$ corresponds to the row operation $B'_{k-1}-\sum_{i=1}^{k-2}\left (\frac{\alpha_1}{\alpha_{i+1}}\right )^{p-1}B'_i$. Therefore
\[\det(B)=\det(B'')=\alpha_2^{p-1}\cdots\alpha_{k-1}^{p-1}\left (\alpha_1^{p-1}+\alpha_k^{p-1}+\sum_{i=2}^{k-1}\frac{\alpha_1^{p-1}\alpha_k^{p-1}}{\alpha_i^{p-1}}\right )=p_k(\alpha_1,\ldots,\alpha_k)\] as desired.
\end{proof}

We conclude the section with the following result concerning the generic Jordan type and maximal Jordan set of $D(1){\downarrow_{E_k}}$.

\begin{thm}\label{T: jtD}
The generic Jordan type of $D(1){\downarrow_{E_k}}$ is $[p-2][p]^{k-1}$ and it has maximal Jordan set $\U{E_k}{D(1)}$ the complement of
\[V(p_k)\cup\bigcup_{i=1}^k V(x_i).
\]
\end{thm}
\begin{proof} If $k=1$, $[X_\alpha]_\B=S(1,1)$ has rank $p-3$ for all $\alpha_1\neq 0$. In this case, the (generic) Jordan type of $D(1){\downarrow_{E_1}}$ is $[p-2]$. Also, $p_1=1$ by our convention and its maximal Jordan set is precisely $\F\backslash\{0\}$. Suppose now that $k\geq 2$. Let $M=D(1){\downarrow_{E_k}}$ and $W=V(p_k)\cup\bigcup_{i=1}^k V(x_i)$. For $0\neq\alpha\in \A^k(\F)$, let $S=[X_\alpha]_\B$. Suppose that the Jordan type of $S$ is given by $[p]^{a_p} [p-1]^{a_{p-1}}\ldots [1]^{a_1}$. By Lemma \ref{rank}, $a_p\leq k-1$ and $a_p=k-1$ if and only if $\alpha_i\neq 0$ for all $i$ and $p_k(\alpha_1,\alpha_2,\ldots,\alpha_k)\neq 0$, i.e., $\alpha\not\in W$. In this case, given $\rank(S^{p-3})=3k-2$ and $\rank(S^{p-2})=2k-2$, we obtain that $a_{p-2}+2a_{p-1}+3a_p=3k-2$ and $a_{p-1}+2a_p=2k-2$. Thus, $a_{p-2}=1$ and $a_{p-1}=0$.  Therefore the Jordan type of $[X_\alpha]_\B$ is $[p-2][p]^{k-1}$ as $\dim M=kp-2$. Furthermore, $[p-2][p]^{k-1}$ is the most dominant Jordan type among all Jordan types of modules of dimension $kp-2$. As such, the maximal and generic Jordan type of $M$ is $[p-2][p]^{k-1}$ and $\U{E}{M}$ is the complement of $W$.
\end{proof}

\begin{rem} When $p=2$ and $k>1$, our calculation shows that the generic Jordan type of $D(1){\downarrow_{E_k}}$ is $[2]^{k-1}$ and it has maximal Jordan set $\U{E_k}{D(1)}$ the complement of $V(p_k)$. This coincides with \cite{Jiang21} using Corollary \ref{C: complement}.
\end{rem}

\section{The simple module $D(p-1)$}
As the setting in the previous section, we assume $p\geq 3$, fix an integer $k\geq 2$ and let $n=kp$. Again, $E_k$ is the elementary abelian $p$-subgroup of $\sym{kp}$ generated by $k$ disjoint $p$-cycles and $p_k$ is the polynomial in Equation \ref{Eq: pk}. In this section, we shall make use of the fact that $D(p-1)$ is the $(p-1)$th exterior power of $D(1)$ to compute the rank variety for $\res{D(p-1)}{E_k}$ when $k\not\equiv 1(\mod p)$. We should like to mention that the rank variety should also be $V(p_k)$ for the case $k\equiv 1(\mod p)$ but our method unfortunately does not apply.

%In this section, we prove the main theorem of this paper. Firstly, we determine the generic Jordan type of $D(p-1)$ with respect to $E_k$. Recall that $D(p-1)\cong \bigwedge^{p-1} D$. Meanwhile, the exterior power of direct sums is as follows: for $\F G$-modules $M$ and $N$,
%\[\bigwedge^m(M\oplus N)\cong \bigoplus^m_{i=0} \bigwedge^i(M)\otimes\bigwedge^{m-i} (N),\]

\begin{lem}\label{jtDp-1}
The generic Jordan type of $\res{D(p-1)}{E_k}$ is $[p]^m$ where $mp=\binom{kp-2}{p-1}$.
\end{lem}
\begin{proof} It is well-known that, for $\F G$-modules $M$ and $N$, we have
\[\bigwedge^m(M\oplus N)\cong \bigoplus^m_{i=0} \bigwedge^i(M)\otimes\bigwedge^{m-i} (N).\] By  Proposition \ref{P: Umax}, we conclude that the generic Jordan type of $\res{D(p-1)}{E_k}\cong \res{\left (\bigwedge^{p-1} D(1)\right )}{E_k}$ is the $(p-1)$th exterior power of the generic Jordan type of $D(1)$.  By Theorem \ref{T: jtD}, the generic Jordan type of $D(1)$ is $[p-2][p]^{k-1}$. Notice that
\begin{align*}
N:=\bigwedge^{p-1}(J_{p-2}\oplus J_p^{\oplus k-1})\cong \bigoplus^{p-1}_{i=0} \bigwedge^{p-1-i} J_{p-2}\otimes \bigwedge^{i}(J_p^{\oplus k-1}).
\end{align*}
Unless $i=0$, $\bigwedge^{i}(J_p^{\oplus k-1})$ is free. However, when $i=0$, $\bigwedge^{p-1} J_{p-2}$ is the zero module. Therefore, $N$ is free. Since the dimension of $D(p-1)$ is $\binom{kp-2}{p-1}$, we get the number $m$ as desired. The proof is now complete.
\end{proof}

%Now given that the maximal or generic Jordan type of $\res{D(p-1)}{E_k}$ is $[p]^m$ for some positive integer $m$.
Since the generic Jordan type of $\res{D(p-1)}{E_k}$ is free, by Corollary \ref{C: complement}, we have $V^\#_{E_k}(D(p-1))=(\U{E_k}{D(p-1)}^c$. By Theorem \ref{T: jtD} and Proposition \ref{P: Umax}, we conclude that
\begin{equation}\label{subset}
V_{E_k}^\#(D(p-1))\subseteq V(p_k)\cup\bigcup_{i=1}^k V(x_i).
\end{equation} More precisely, we have the following lemma. In the proof, we need the modular branching rule. Since we do not require it elsewhere, we refer the readers to \cite{kleshchev} {for the necessary details}.

\begin{lem}\label{subseteq} $V_{E_k}^\#(D(p-1))\subseteq V(p_k)$.
\end{lem}
\begin{proof}
Given Equation \ref{subset}, we just need to show that, for any $1\leq i\leq k$,
\begin{equation*}
V_{E_k}^\#(D(p-1))\cap V(x_i)\subseteq \bigcup_{j\neq i} V(x_j,x_i) \subseteq V(p_k).
\end{equation*}
The second inclusion is clear since, if any two of the coordinates of a point are $0$, then $p_k$ vanishes at this point. Showing the first inclusion is equivalent to showing that, for any $\alpha=(\alpha_1,\dots,\alpha_k)\in V_{E_k}^\#(D(p-1))$ with $\alpha_i=0$ for some $i$, there exists $j\neq i$ such that $\alpha_j=0$. Let $\alpha$ be such a point. By Lemma \ref{L: symmetry}, we can further assume that $i=k$ and let $\alpha'=(\alpha_1,\ldots,\alpha_{k-1})$. Since \[u_\alpha=u_{\alpha'}=1+\sum_{i=1}^{k-1}\alpha_i(g_i-1)\in \F\sym{kp-1},\] we have $\res{D(p-1)}{\langle u_\alpha\rangle}=\res{(\res{D(p-1)}{\sym{kp-1}})}{\langle u_{\alpha'}\rangle}$ and therefore $\alpha'\in V^\#_{E_{k-1}}(D(p-1))$. By the modular branching rule \cite[Theorem 11.2.7]{kleshchev}, since the only one good node in $(kp-p+1,1^{p-1})$ is the $(1,kp-p+1)$-node (see the diagram below indicating the necessary residues and the good node), we have $\res{D(p-1)}{\sym{kp-1}}\cong D^{(kp-p,1^{p-1})}$. \[
\begin{ytableau}
0&1&  \none[\dots]&\scriptstyle p-1&0&  \none[\dots]&\scriptstyle p-1&*(yellow)0&\none[1]
\cr \scriptstyle p-1 &\none[0]
\cr \scriptstyle p-2
\cr \none[\vdots]
\cr 1
\cr \none [0]
\end{ytableau}\] Since $(kp-p,1^{p-1})$ is a hook of size coprime to $p$, by \cite[Theorem 2]{Peel} and \cite[Theorem 11.1]{James78}, we have $D^{(kp-p,1^{p-1})}\cong S^{(kp-p,1^{p-1})}$. By \cite[Theorem 7.3]{Wildon10}, $S^{(kp-p,1^{p-1})}$ has a vertex a Sylow $p$-subgroup $P$ of $\sym{kp-2p}$ and has a trivial $\F P$-source. Therefore, using Mackey's formula, \[\res{D(p-1)}{E_{k-1}}\mid\res{\ind{\F_P}{\sym{kp-1}}}{E_{k-1}}\cong\bigoplus_{s\in E_{k-1}\backslash \sym{kp-1}/P} \ind{\res{{}^s\F_P}{E_{k-1}\cap sPs^{-1}}}{E_{k-1}}
\] where each $E_{k-1}\cap sPs^{-1}$ is a proper subgroup of $E_{k-1}$. Consequently, we have:
\[\alpha'\in V_{E_{k-1}}^\#(D(p-1))\subseteq \bigcup_{s\in E_{k-1}\backslash \sym{kp-1}/P} V_{E_{k-1}}^\#(\ind{\res{{}^s\F_P}{E_{k-1}\cap sPs^{-1}}}{E_{k-1}}),
\]
and there exist some $s\in E_{k-1}\backslash \sym{kp-1}/P$ such that $\alpha' \in V_{E_{k-1}}^\#(\ind{\res{{}^s\F_P}{E_{k-1}\cap sPs^{-1}}}{E_{k-1}})$. Let $E'=E_{k-1}\cap sPs^{-1}$. Since $P\leq \sym{(k-2)p}$, we have $sPs^{-1} \leq\sym{\{s(1),s(2),\dots,s((k-2)p)\}}$. Therefore $sPs^{-1}$ permutes only $(k-2)p$ numbers and $E'$ is properly contained in $E_{k-1}$. We claim that there exist $1\leq j\leq k-1$ such that, for any $\prod_{i=1}^{k-1} g_i^{q_i}\in E'$, $q_j=0$. If not, then for any $1\leq j\leq k-1$, there exists, some $\prod_{i=1}^{k-1} g_i^{q_i}\in E'$ such that $q_j\neq 0$. This shows that $E'$ permutes all $(k-1)p$ numbers, which is a contradiction. Since $E'$ avoids the generator $g_j$ completely, by Lemma \ref{L: rkinduction}, we have $w_{ij}=0$ for $1\leq i\leq k-1$ in that statement, i.e. $\alpha_j=0$. The proof is now complete.
\end{proof}

Now we are ready to state and prove the main theorem in this section.

\begin{thm}\label{T: main thm}
If $k\not\equiv 1 (\mod p)$, then
\[V_{E_k}^\#(D(p-1))=V(p_k).\]
\end{thm}
\begin{proof} From Lemma \ref{subseteq}, we knew that $W:=V_{E_k}^\#(D(p-1))\subseteq V(p_k)$ and \begin{align*}
\dim_\F D(p-1)=\binom{kp-2}{p-1}&=\frac{(kp-2)(kp-3)\ldots(k-1)p}{(p-1)!}\\
&=\frac{(kp-2)(kp-3)\ldots((k-1)p+1)}{(p-1)!}\cdot (k-1)p.
\end{align*}
Since $k\not \equiv 1(\mod p)$, $p$ divides $\dim_\F {D(p-1)}$ but $p^2$ does not. Let $r=\dim W$. By Theorem \ref{T: dimension}, we have $p^{k-r}\mid \dim_\F D(p-1)$ and so $r\geq k-1$. On the other hand, $\dim {V(p_k)}= k-1$. Therefore $r=k-1$. When $k\geq 3$, by Lemma \ref{irred}, $V(p_k)$ is irreducible and we must have $W= V(p_k)$ by Theorem \ref{T: irred variety dim}.

Suppose now that $k=2$. Notice that $V(p_2)$ is reducible and we have a factorisation
\[X^{p-1}+Y^{p-1}=\prod_{r\in I} (X-rY),
\]
where $I$ is the set of all $(p-1)$th roots of $-1$ in $\F$. We have $V(p_2)=\bigcup_{r\in I}V(X-rY)$ and this gives the irreducible components of $V(p_2)$. On the other hand, $\dim W=1=\dim V$. Then $W=\bigcup_{r\in I'} V(X-rY)$ for some $\emptyset\neq I'\subseteq I$. Let $\lambda$ be a primitive $(2p-2)$th root of 1 in $\F$ where $\langle\lambda\rangle$ is a subgroup of $\F_{p^2}^\times$. We have $I=\{\lambda^1,\lambda^3,\dots,\lambda^{2p-3}\}$ and $\{\lambda^2,\lambda^4,\dots,\lambda^{2p-2}\}=\F_p^{\times}$. Let $r'=\lambda^{2i+1}\in I'$. By Lemma \ref{L: symmetry}, for any $r=\lambda^{2j+1}\in I$, we have \[(r,1)=(\lambda^{2j-2i}\lambda^{2i+1},1)=\lambda^{2j-2i}\cdot (r',1)\in W\] where we have considered $\lambda^{2j-2i}$ belongs to the first component in $\F_p^\times\wr \sym{2}$. As such, $I'=I$ and we have $W=V(p_2)$ as required.
\end{proof}

\begin{rem} The only obstruction for applying the proof of Theorem \ref{T: main thm} to the remaining case; namely, $k\equiv 1(\mod p)$, is that we could not deduce that $\dim V^\#_{E_k}(D(p-1))=k-1$. {If the dimension were indeed $k-1$, the rest of the proof follows. Also, Corollaries \ref{C: up bound pk}, \ref{C: variety kp-p-1} and \ref{C: comp} below would also follow at once. }
\end{rem}

We conclude the paper with corollaries following our result.

\begin{cor}\label{C: up bound pk} For $k\not\equiv 1(\mod p)$, we have $d_{V(p_k)}\leq \binom{kp-2}{p-1}$.
\end{cor}

\begin{cor}\label{C: variety kp-p-1} If $k\not\equiv 1(\mod p)$, then $V^\#_{E_k}(D(kp-p-1))=V(p_k)$.
\end{cor}
\begin{proof} By \cite[Theorem 8.15]{James78}, $S^\lambda\otimes \sgn\cong (S^{\lambda'})^*$ where $\lambda'$ is the conjugate partition of $\lambda$. Since $D(p-1)$ appears as the head of $S^{(kp-p+1,1^{p-1})}$ and $D(kp-p-1)$ appears as the socle of $S^{(p,1^{kp-p})}$, we have $D(p-1)\otimes \sgn\cong D(kp-p-1)$. Therefore, by Theorem \ref{T: basic rank}(iii), the restriction of both modules $D(p-1)$ and $D(kp-p-1)$ to $E_k$ have the same rank variety.
\end{proof}

\begin{cor}\label{C: comp} Suppose that $k\not\equiv 1 (\mod p)$. The complexities of the simple modules $D(p-1)$ and $D(kp-p-1)$ are $k-1$.
\end{cor}
\begin{proof} Since $p\geq 3$, the $p$-rank of an elementary abelian $p$-subgroup $E$ of $\sym{kp}$ is strictly less than $k$ unless $E$ is conjugate to $E_k$. Since the dimension of $V^\#_{E_k}(D(p-1))$ is $k-1$ and the rest are not more than $k-1$, the maximal value must be $k-1$. By Equation \ref{Eq: complexity} and Theorem \ref{T: basic rank}(ii), the complexity of $D(p-1)$ must be $k-1$. The same holds for $D(kp-p-1)$ using Corollary \ref{C: variety kp-p-1}.
\end{proof}

\begin{cor} Let $\lambda_1,\ldots,\lambda_{p-1}$ be the $(p-1)$th roots of $-1$ in $\F$. When $k=2$, the $\F\sym{2p}$-module $D(p-1)$ restricted to $E_2$ decomposes into $Q\oplus\bigoplus_{j=1}^{p-1} N_j$ such that $Q$ is projective, for each $1\leq j\leq p-1$, $N_j$ is projective-free and $V^\#_{E_2}(N_j)=V(X-\lambda_jY)$.
\end{cor}
\begin{proof} By Theorem \ref{T: main thm}, when $k=2$, the connected components of the projective variety $\overline{V^\#_{E_2}(D(p-1))}$ are singleton points $(\lambda_j:1)$ one for each $1\leq j\leq p-1$. As such, by Theorem \ref{T: basic rank}(iv), there are summands $Q,N_1,\ldots,N_{p-1}$ of $\res{D(p-1)}{E_2}$ with the desired property.
\end{proof}

\begin{rem} By \cite{Danz07,DG15}, it is known that the Green vertices of $D(p-1)$ are the Sylow $p$-subgroups of $\sym{kp}$. In the case of $k\not\equiv 1(\mod p)$, Lemma \ref{irred}, Theorems \ref{T: Green} and \ref{T: main thm} confirm that $D(p-1)$ has a Green vertex containing $E_k$ as $V(p_k)$ is not contained in the union of the proper base subspaces. Moreover, the same holds for $D(kp-p-1)$.
\end{rem}

%\KJ{Find a point in $V(p_k)$ but not in the union of proper base subspaces.}

\end{document}